\documentclass[12pt]{amsart}
\usepackage{amssymb,latexsym,amsmath,doublespace,mathrsfs}
\usepackage[all]{xy}

\newcommand{\Cf}{\mathrm{C}}   
\newcommand{\Cb}{\mathrm{C}_{\mathrm{b}}} 

\newcommand{\cC}{\mathcal{C}}

\newcommand{\cK}{\mathcal{K}}     
\newcommand{\cM}{{\mathcal{M}}}   
\newcommand{\cR}{{\mathcal{R}}}  
\newcommand{\T}{{\mathcal{T}}}

\newcommand{\cZ}{\mathcal{Z}}    
\newcommand{\C}{\mathbb{C}}        
\newcommand{\R}{\mathbb{R}}        
\newcommand{\N}{\mathbb{N}}        
\newcommand{\ot}{\otimes}

\newcommand{\cf} {\textrm{cf.\/}\ }
\newcommand{\cst}{C$^*$}
\newcommand{\Cs}{$C^*$-algebra}
\newcommand{\sh}{{$^*$-ho\-mo\-mor\-phism}}
\newcommand{\Ker}{\mathrm{Ker}}
\newcommand{\ep}{\varepsilon}
\newcommand{\Cu}{\mathrm{Cu}}

\newtheorem{thm}{Theorem}[section]
\newtheorem{cor}[thm]{Corollary}
\newtheorem{lemma}[thm]{Lemma}
\newtheorem{prop}[thm]{Proposition}

\theoremstyle{definition}
\newtheorem{definition}[thm]{Definition}

\newtheorem{rem}[thm]{Remark}

\numberwithin{equation}{section}
\begin{document}

\title{Central sequence $C^*$-algebras and tensorial absorption of the Jiang-Su algebra}
\author{Eberhard Kirchberg  and  Mikael R{\o}rdam}

\thanks{This research was supported by the Danish National Research Foundation
(DNRF) through the Centre for Symmetry and Deformation at University
of Copenhagen. The first named author thanks the Centre for Symmetry and Deformation for hosting him as a visiting professor during the fall of 2011, where this work was initiated.}
\date{November 16, 2012}
%
%
%
%

\begin{abstract}
We study properties of the central sequence algebra $A_\omega \cap A'$ of a \Cs{} $A$, and we present an alternative approach to a recent result of Matui and Sato. They prove that every unital separable simple nuclear \Cs, whose trace simplex is finite dimensional, tensorially absorbs the Jiang-Su algebra if and only if it has the strict comparison property. We extend their result to the case where the extreme boundary of the trace simplex is closed and of finite topological dimension. We are also able to relax the assumption on the \Cs{} of having the strict comparison property to a weaker property, that we call \emph{local weak comparison}. Namely, we prove that a unital separable simple nuclear \Cs, whose trace simplex has finite dimensional closed extreme boundary,  tensorially absorbs the Jiang-Su algebra if and only if it has the local weak comparison property.

We can also eliminate the nuclearity assumption, if instead we assume the (SI) property of Matui and Sato, and, moreover,  that each II$_1$-factor representation of the \Cs{} is a McDuff factor.
\end{abstract}

\maketitle

\tableofcontents

\section{Introduction} 

\noindent In  the quest to classify simple separable nuclear \Cs s, pioneered by George Elliott, it has become necessary to invoke some regularity property of the \Cs, without which classification---at least by simple $K$-theoretical invariants---is known to be impossible, see for example \cite{Toms.example} and \cite{Ror.simple}. There are three such, seemingly very different, regularity properties of particular interest:  tensorial absorption of the so-called Jiang-Su algebra $\cZ$, also called $\cZ$-stability; finite nuclear dimension (after Wilhelm Winter); and strict comparison of positive elements (after Bruce Blackadar). The latter can be reformulated as an algebraic property of the Cuntz semigroup, called almost unperforation. Toms and Winter have conjectured that these three fundamental properties are equivalent for all separable, simple, nuclear \Cs s. The Toms--Winter conjecture is known to hold extensively. All $\cZ$-stable \Cs s have strict comparison, \cite{Ror.Z.absorb}. Winter has shown in \cite{Winter.Z} that finite nuclear dimension ensures $\cZ$-stability. 

Last year, in a remarkable paper by Matui and Sato, \cite{Matui.Sato}, it was shown that the strict comparison property is equivalent to $\cZ$-stability for all unital, separable, simple, nuclear \Cs s whose trace simplex has finite dimension, i.e., its extreme boundary is a finite set. Matui and Sato employed  new techniques involving the central sequence \Cs, $A_\omega \cap A'$. They introduced two new properties of the central sequence algebra: property (SI) and "excision in small central sequences". They show that both properties are equivalent to strict comparison and to $\cZ$-absorption under the above given conditions. 

The approach of Matui and Sato has some similarities with the approach used by Phillips and the first named author in their proof in \cite{KirPhi.O2embedding} that purely infinite simple separable nuclear \Cs s tensorially absorb the Cuntz algebra $\mathcal{O}_\infty$. The latter result likewise relied on a detailed study of the central sequence \Cs. Indeed, it was shown in \cite{KirPhi.O2embedding} that if $A$ is a unital purely infinite simple separable nuclear \Cs, then $A_\omega \cap A'$ is a purely infinite and simple \Cs. In particular, there is a unital embedding of $\mathcal{O}_\infty$ into $A_\omega \cap A'$, which implies that $A \cong A \otimes \mathcal{O}_\infty$. 

The first named author studied the central sequence algebra in \cite{Kir.AbelProc}. We rely on, and further develop, the techniques of that paper to arrive at the main result stated in the abstract. 

Our weakened comparability assumption is automatically satisfied by any unital simple \Cs{} satisfying Winter's $(m,\bar{m})$-pureness condition from \cite{Winter.Z}, cf.\ Lemma~\ref{cor:nuclear.case}. The main theorem from Winter's paper, \cite{Winter.Z}, concerning \Cs s with locally finite nuclear dimension therefore follows from our Corollary \ref{cor:nuclear.case} in the case where the extreme boundary of the trace simplex of the given \Cs{} is closed and has finite topological dimension. 

Let us describe some of the main  ingredients in Matui and Sato's paper \cite{Matui.Sato}, that we shall build on and further develop. It is well-known that if $A$ is a separable unital \Cs{}, then $A$ is $\cZ$-absorbing, i.e., $A \cong A \otimes \cZ$ if and only if there is a unital \sh{} from $\cZ$ into the central sequence algebra $A_\omega \cap A'$ for some ultrafilter $\omega$ on $\N$. It follows from \cite[Corollary 1.13]{Kir.AbelProc} that if there is a unital \sh{}  from a separable unital \Cs{} $D$ into $A_\omega \cap A'$ , then there is a unital \sh{} from the maximal infinite tensor power $\bigotimes_{n=1}^\infty D$ to $A_\omega \cap A'$. It was shown by Dadarlat and Toms in \cite{DadarlatToms.Z.stability} that if $D$ contains a unital sub-\Cs{} which is sub-homogeneous and has no characters, then $\cZ$ embeds unitally into the infinite tensor power $\bigotimes_{n=1}^\infty D$. It follows that a unital separable \Cs{} $A$ absorbs the Jiang-Su algebra if and only if there is a unital \sh{} from a sub-homogeneous \Cs{} without characters into $A_\omega \cap A'$. One example of such a \Cs{} is the dimension drop \Cs{} $I(2,3)$, which consists of all continuous functions $f \colon [0,1] \to M_2 \otimes M_3$ such that $f(0) \in M_2 \otimes \C$ while $f(1) \in \C \otimes M_3$.

It was shown by Sato in \cite{Sato.1104} that the natural \sh{} $A_\omega \cap A' \to N^\omega \cap N'$ is surjective, when $A$ is a unital separable nuclear \Cs, $\tau$ is a faithful tracial state on $A$, and $N$ is the finite von Neumann algebra arising from $A$ via this trace. The kernel, $J_\tau$, of the \sh{} $A_\omega \cap A' \to N^\omega \cap N'$  consists of those elements $a$ in $A_\omega \cap A'$ for which $\|a\|_{2,\tau} = 0$ (see Definition \ref{def:seminorms}). In other words, $(A_\omega \cap A')/J_\tau$ is a finite von Neumann algebra when the conditions above are satisfied. If $N$ is a McDuff factor, then $N^\omega \cap N'$ is a II$_1$ von Neumann algebra and there is, in particular, a unital \sh{} $M_2 \to (A_\omega \cap A')/J_\tau$. One would like to lift this \sh{} to a unital \sh{} $I(2,3) \to A_\omega \cap A'$, since this will entail that $A$ is $\cZ$-absorbing. Matui and Sato invented the property, called (SI), to solve this problem. Property (SI) is a (weak) comparability property of the central sequence algebra. In order to prove that property (SI) holds, Matui and Sato introduce another property of the central sequence algebra, called \emph{excision in small central sequences}. They show that excision in small central sequences can be obtained from strict comparability and from nuclearity; and they remark that property (SI) holds if the identity map on the \Cs{} can be excised in small central sequences. 

We give an elementary proof of  Sato's result on surjectivity of the map $A_\omega \cap A' \to N^\omega \cap N'$, and we show that surjectivity holds without assuming that $A$ is nuclear, see Theorem \ref{thm:M-omega}. The proof uses, at least implicitly, that the ideal $J_\tau$ mentioned above is a so-called $\sigma$-ideal, see Remark \ref{rem:J_tau}. In Section \ref{sec:trace.kernel.ideal} we consider the ideal $J_A$ of $A_\omega$ consisting of those elements that are represented by sequences in $\ell^\infty(A)$ with uniformly vanishing trace. When $A$ has infinitely many extremal traces, then $(A_\omega \cap A')/J_A$ is no longer a von Neumann algebra.  The main effort of this paper is to verify that, nonetheless, there is a unital \sh{} $M_2 \to (A_\omega \cap A')/J_A$ under suitable conditions on $A$. As in the original paper by Matui and Sato, one can then use property (SI) to lift this unital \sh{} to a unital \sh{} $I(2,3) \to A_\omega \cap A'$, and in this way obtain $\cZ$-stability of $A$.

We have learned that Andrew Toms, Stuart White, and Wilhelm Winter independently have obtained results very similar to ours, see \cite{Toms.Winter.Z.absorb}. Also Yasuhiko Sato has independently obtained a similar result, see \cite{Sato.1209}. We hope, nonetheless, that the methods of this paper, which offer a new view on the use of central sequence \Cs s, as well as the somewhat higher generality of our results, will make this paper worthwhile to the reader. 

In the following section we give an overview of our main results and we outline the methods we are using.

\section{Preliminaries and overview}
\label{sec:intro}
\noindent We present here some of the main ideas of our paper, and we also give the most important definitions that will be used throughout the paper. First we discuss some notions of comparability in \Cs s.

A \Cs{} $A$ is said to have \emph{strict comparison} if comparison of positive elements in matrix algebras over $A$ is determined by traces. In more detail, 
for all $k \ge 1$ and for all positive elements  $a,b \in  M_k(A)$, if  $d_\tau (a) < d_\tau (b)$
 for all $2$-quasi-traces $\tau$ on $A$, then
 $\langle a \rangle \leq \langle b \rangle$ in the Cuntz semigroup $\Cu(A)$. The latter holds if  
there exists a sequence $(r_n)_{n \ge 1}$ in $M_k(A)$ such that $r_n^*br_n \to a$, and when this holds we write $a \precsim b$. We remind the reader that $d_\tau$ is the \emph{dimension function} associated to $\tau$, and it is defined to be $d_\tau(a) = \lim_{n\to\infty}\tau(a^{1/n})$. 
 If $A$ is exact (and unital), then all $2$-quasi-traces are traces, thanks to a theorem of Haagerup, \cite{Haagerup.1991.notes}, and we can then replace "unital $2$-quasi-traces" with "tracial states" in the definition of strict comparison. The set of  unital $2$-quasi-traces on $A$ will be denoted by $QT(A)$, and the set of tracial states on $A$ by $T(A)$.
It was shown in \cite{Ror.Z.absorb} that strict comparison for $A$ is equivalent to saying that $\Cu(A)$ is \emph{almost unperforated}, i.e., if $x,y \in \Cu(A)$ are such that $(n+1)x \le ny$ for some $n \ge 1$, then $x \le y$. 

We introduce a weaker comparison property of $A$ as follows:
\begin{definition}\label{def:weak.comparison}
Let $A$ be a unital, simple, and stably finite \Cs.
We say that $A$ has \emph{local weak comparison}, if 
there is a 
constant $\gamma(A)\in [1,\infty)$
such that the following holds for all positive elements $a$ and $b$ in $A$: If  
$$\gamma(A) \cdot \sup_{\tau\in QT(A)} d_\tau(a) <  \inf_{\tau\in QT(A)}  d_\tau(b),$$
then 
$\langle a \rangle  \leq  \langle b \rangle$ in the Cuntz semigroup 
$\Cu(A)$  of $A$.

If $M_n(A)$ has local weak comparison for all $n$, and $\sup_n \gamma(M_n(A)) < \infty$,  then we say that $A$ has \emph{weak comparison}.
\end{definition}

\smallskip
\noindent It is clear that strict comparison implies (local) weak comparison.

\begin{definition} Let $A$ be a \Cs{} and let $1 \le \alpha < \infty$. We say that $A$ has the \emph{$\alpha$-comparison property} if for all $x,y \in \Cu(A)$ and all integers $k, \ell \ge 1$ with $k > \alpha \ell$, the inequality $kx \le \ell y$ implies $x \le y$. 
\end{definition}

\noindent The lemma below follows easily from \cite[Proposition 3.2]{Ror.Z.absorb}.

\begin{lemma} A \Cs{} $A$ has $\alpha$-comparison for some $1 \le \alpha < \infty$ if for all $x,y \in \Cu(A)$ with $x \propto y$ (i.e., $x \le ny$ for some $n$), the inequality $\alpha f(x) < f(y)=1$ for all $f \in S(\Cu(A),y)$ (the set of additive, order preserving maps $\Cu(A) \to [0,\infty]$ normalized at $y$) implies $x \le y$.
\end{lemma}

\noindent
Our $\alpha$-comparison property is related to Winter's $m$-comparison property,  see \cite[Section 2]{Winter.Z}. We also have the following fact, the proof of which essentially is contained in \cite{Winter.Z}. 

\begin{lemma} \label{lm:pure}
Let $A$ be a simple unital \Cs{} with $QT(A) \ne \emptyset$. If $A$ has strong tracial $m$-comparison (in the sense of Winter, \cite[Definition 2.1]{Winter.Z}) for some $m \in \N$, then $A$ has weak comparison and local weak comparison (in the sense of Definition \ref{def:weak.comparison}). In particular, if $A$ is $(m,\bar{m})$-pure for some $m,\bar{m} \in \N$ (in the sense of Winter, \cite[Definition 2.6]{Winter.Z}), then $A$ has weak comparison and  local weak comparison.
\end{lemma}

\begin{proof}   Suppose that $A$ has strong tracial $m$-comparison. We show that $M_n(A)$ has local weak comparison with $\gamma = \gamma(M_n(A)) = m+1$ for all $n$. It suffices to verify the local weak comparison for arbitrary \emph{contractions} $a,b \in M_n(A)$. Suppose that
$$\gamma \cdot \sup_{\tau\in QT(A)} d_\tau(a) <  \inf_{\tau\in QT(A)}  d_\tau(b).$$
Then, in particular, $\gamma \cdot d_\tau(a) < d_\tau(b)$ for all $\tau \in QT(A)$. 
Arguing as in the proof of \cite[Proposition 2.3]{Winter.Z} (and with $g_{\eta,\ep}$ as defined in \cite[1.4]{Winter.Z}), one obtains that for each $\ep >0$ there exists $\eta > 0$ such that
$$d_\tau((a-2\ep)_+) = d_\tau(g_{2\ep,3\ep}(a)) < \frac{1}{m+1} \cdot \tau(g_{\eta,2\eta}(b)),$$
for all $\tau \in QT(A)$. Strong tracial $m$-comparison thus implies that $(a-2\ep)_+ \precsim g_{\eta,2\eta}(b) \precsim b$ for all $\ep > 0$. This, in turns, shows that $a \precsim b$, or, equivalently, that $\langle a \rangle \le \langle b \rangle$ in $\Cu(A)$. 

It is shown in \cite[Proposition 2.9]{Winter.Z}, that if $A$ is $(m,\bar{m})$-pure, then $A$ has strong tracial $\tilde{m}$-comparison for some $\tilde{m} \in \N$. The second claim of the lemma therefore follows from the first. 
\end{proof}

\noindent It is easy to see that our local weak comparison property is weaker than strict comparison within the class of \Cs s of the form $M_n(C(X))$, where $X$ is a finite dimensional compact Hausdorff space. (A more subtle arguments shows that also our weak comparison is weaker than strict comparison within this class.)
It is an open problem if (local) weak comparison, Winter's $m$-comparison, our $\alpha$-comparison, and strict comparison all agree for \emph{simple} \Cs s, cf.\ Corollary \ref{cor:nuclear.case}.

We introduce  in Section \ref{sec:trace.kernel.ideal}, for every  $p\in [1,\infty)$, the following semi-norms  
on $A$ and its ultrapower $A_\omega$ (associated with a free ultrafilter $\omega$ on $\N$):
$$
\| \,a \, \|_{p,\tau} := \tau((a^*a)^{p/2})^{1/p}\, \qquad 
\| \,a \, \|_{p} := \sup _{\tau\in T(A)} \, \| \,a \, \|_{p,\tau}, \quad a \in A;
$$
and 
$$\|  \pi_\omega (a_1,a_2,\ldots) \|_{p,\omega}  := 
\lim_{n \to \omega} \| a_n \|_{p},$$
where $\pi_\omega \colon \ell^\infty(A) \to A_\omega$ denotes the quotient mapping. 
Let $J_A$ be the closed two-sided ideal in $A_\omega$, consisting of all elements $a \in A_\omega$ such that $\|a\|_{p,\omega} = 0$ for some (and hence all) $p \in [1,\infty)$. We call $J_A$ the \emph{trace-kernel ideal}. It will be discussed in more detail in Section \ref{sec:trace.kernel.ideal}, and so will the norms defined above. 

The \emph{central sequence algebra} $A_\omega \cap A'$ (for $A$ unital) will be denoted by $F(A)$. As explained very briefly in Section \ref{sec:sequence.algebras}, and in much more detail in \cite{Kir.AbelProc}, one can define $F(A)$ in a meaningful way also for non-unital \Cs s, such that $F(A)$, for example, always is unital. 
We reformulate the Matui--Sato definition of "excision in small central sequences" (see Definition \ref{def:Matui.Sato.excision}) to the context of ultraproducts as follows:

\begin{definition}\label{def:omega.related.excision}
A  completely positive map $\varphi \colon A\to A$ can be 
\emph{excised in small central sequences} if, for all $e,f\in F(A)$
with  $e\in J_A$ and $\sup_n \| 1- f^n \|_{2,\omega}<1$,
there exists $s\in A_\omega$ with $fs=s$ and $s^*as=\varphi (a)e$ for all $a\in A$.
\end{definition}

\noindent
We show in Lemma \ref{lem:excision-eq} that our Definition \ref{def:omega.related.excision} above is equivalent to the one of Matui and Sato.   One can replace $\| \,\cdot\, \|_{2,\omega}$ in  Definition 
\ref{def:omega.related.excision} by 
 $\| \,\cdot\, \|_{p,\omega}$ with any $p\in [1,\infty)$, even if  $QT(A) \ne T(A)$, cf.\ the comments below Definition \ref{def:seminorms}. 

Since $J_A$
is a so-called $\sigma$-ideal of $A_\omega$ (see
Definition \ref{def:sigma.ideals.etc}) one can replace the condition $e\in F(A)\cap J_A$
with the formally weaker condition $e\in J_A$. The assumption  that $f\in F(A)$, however, cannot be weakened.

We show in Proposition \ref{prop:from.comparison.to.excision} that every nuclear completely positive map $A \to A$ can be excised in small central sequences provided that $A$ is unital, separable, simple, and stably finite with the local weak comparison property. This result strengthen the results of Sections 2 and 3 in \cite{Matui.Sato}.  We prove Proposition \ref{prop:from.comparison.to.excision} without using nuclearity of $A$, or other additional requirement. In particular, this proposition can be proved without using
Section 3 of \cite{Matui.Sato} (including Lemma 3.3 of that Section, which involves deep results about von Neumann algebras).

We also reformulate Matui and Sato's property (SI) in the language of central sequence algebras.

\begin{definition}\label{def:property(SI)}
A unital simple  \Cs{} is said to have property (SI)  if
for all positive contractions $e,f\in F(A)$,
with $e\in J_A$ and $\sup_n \| 1-f^n\|_{2,\omega} <1$,
there exists $s\in F(A)$ with $fs=s$ and $s^*s=e$.
\end{definition}

\noindent  We shown in Lemma \ref{lem:SI-eq} that our Definition \ref{def:property(SI)} is equivalent to the original definition of Matui and Sato, see Definition \ref{def:Matui.Sato.SI}.
Also in Definition \ref{def:property(SI)} one can replace the semi-norm $\| \, \cdot\, \|_{2,\omega}$
 by $\| \, \cdot\, \|_{p,\omega}$
for any $p\in [1,\infty)$.

Matui and Sato proved in \cite{Matui.Sato} that if the identity map $\mathrm{id}_A$ on a unital simple \Cs{} $A$ can be excised in small central sequences, then $A$ has property (SI).  Property (SI) for non-nuclear \Cs s is mysterious;  we cannot see any connection
between comparison properties of $A$ itself and (SI) if $A$ is \emph{not} nuclear.

The importance of property (SI) is expressed in the Proposition 
\ref{prop:(SI).and.unital.image.of.I(2,3).in.F(A).mod.JA}.  This proposition, which implicitly is included in Matui and Sato's paper, \cite{Matui.Sato}, says that if $A$ has property (SI), and if $A$ is simple and unital, then the existence of a unital \sh{} $M_2 \to F(A)/(J_A \cap F(A))$ implies the existence of a unital \sh{} $I(2,3) \to F(A)$. The latter, in turns, implies that $A \cong A \otimes \cZ$ if $A$ is separable. 

From the point of view of logical completeness it would be desirable to prove that
property (SI) together with the existence of a unital \sh{} 
$\psi\colon B\to F(A)/( J_A \cap F(A))$
from any  sub-homogeneous unital \Cs{} $B$ without characters,
implies the existence of a unital \sh{} $\psi\colon C\to F(A)$
for some (possibly other)  unital recursively sub-homogeneous \Cs{} $C$
without characters. This would imply that $A\cong A\otimes \mathcal{Z}$, because $\bigotimes_{n=1}^\infty C$ contains a unital copy of $\cZ$, by  Dadarlat--Toms, \cite{DadarlatToms.Z.stability}, and because there is a unital \sh{} from  $\bigotimes_{n=1}^\infty C$ into $F(A)$, 
by \cite[Corollary 1.13 and Proposition 1.14]{Kir.AbelProc}.

It is therefore an important (open) question if $F(A)/(J_A \cap F(A))$ contains a 
sub-homogeneous unital \Cs{} that admits no characters whenever  $F(A)/(J_A \cap F(A))$  itself  has no characters.

In Section \ref{sec:partial.T(A).closed} we consider the extreme boundary $\partial T(A)$ and its weak$^*$ closure, denoted $bT(A)$, of the trace simplex $T(A)$ of a unital \Cs{} $A$. We study the ultra-power 
$$\T_\omega \colon A_\omega \to \Cf (\partial T(A))_\omega$$
of the completely positive map $\T \colon A\to \Cf (\partial T(A))$, given
by $\T(a)(\tau):=\tau (a)$ for $\tau\in T(A)$ and $a\in A$. The trace-kernel ideal $J_A$, mentioned earlier, consists of all elements $a\in A_\omega$ with $\T_\omega (a^*a)=0$.

The Choquet simplex $T(A)$ is a Bauer simplex
if and only if $\T_\omega$ maps $\mathrm{Mult}(\T_\omega  |_{F(A)})$ \emph{onto} $\Cf (b T(A))_\omega$, where $\mathrm{Mult}(\Phi  |_{F(A)})$ denotes the multiplicative domain defined by
$$\mathrm{Mult}(\T_\omega  |_{F(A)}) := \big\{ a\in F(A):  \T_\omega (a^*a)=\T_\omega (a)^*\T_\omega (a), \, \T_\omega(aa^*)=\T_\omega (a)\T_\omega (a)^* \big\}.$$ 
We do not use nuclearity of $A$ to obtain this fact. 
As the restriction of $\T_\omega$ to $\mathrm{Mult}(\T_\omega  |_{F(A)})$ is a $^*$-epimorphism, this allows us to find, for each positive function $f\in \Cf (\partial T(A))$
with $\| f\|=1$ and any compact subset $\Omega \subset A$ 
and $\varepsilon>0$,  a positive contraction
$a\in A$ with $\| [a,b] \|<\varepsilon$    
and $\| \T (aba)-f^2\T (b) \| < \varepsilon$ for $b\in \Omega$.

In Section \ref{sec:partial.T(A).fin.dim}, we shall use the existence of such elements $a$ as above, together with the assumption that $\partial T(A)$ has finite topological dimension, to assemble a collection of unital completely positive "\emph{locally} tracially almost order zero maps" $M_k \to A$ into a single unital completely positive  "\emph{globally} tracially almost order zero map" $M_k \to A$. In this way be obtain our main result, Theorem \ref{thm:(SI).implies.Z-absorbtion}.

\medskip


\section{The central sequence algebra}
\label{sec:sequence.algebras}

\noindent To each free filter $\omega$ on the natural numbers and to each \Cs{} $A$ one can associate the ultrapower $A_\omega$ and the central sequence \Cs{} $F(A)=A_\omega \cap A'$ as follows. Let $c_\omega(A)$ denote the closed two-sided ideal of the \Cs{} $\ell^\infty(A)$  given by
$$c_\omega(A) = \big\{(a_n)_{n \ge 1} \in \ell^\infty(A) \mid \lim_{n \to \omega} \|a_n\| = 0\big\}.$$
We use the notation $\lim_{n \to \omega} \alpha_n$ (and sometimes just $\lim_\omega \alpha_n$) to denote the limit of a sequence $(\alpha_n)_{n \ge 1}$ along the filter $\omega$ (whenever the limit exists).

The ultrapower $A_\omega$ is defined to be the quotient \Cs{} $\ell^\infty(A)/c_\omega(A)$; and we denote by $\pi_\omega$ the quotient mapping $\ell^\infty(A) \to A_\omega$. Let $\iota \colon A \to \ell^\infty(A)$ denote the "diagonal" inclusion mapping $\iota(a) = (a,a,a, \dots) \in \ell^\infty(A)$, $a \in A$; and put $\iota_\omega = \pi_\omega \circ \iota \colon A \to A_\omega$. Both mappings $\iota$ and $\iota_\omega$ are injective. We shall often suppress the mapping $\iota_\omega$ and view $A$ as a sub-\Cs{} of $A_\omega$. The relative commutant, $A_\omega \cap A'$, then consists of elements of the form $\pi_\omega(a_1,a_2,a_3, \dots)$, where $(a_n)_{n \ge 1}$ is a bounded asymptotically central sequence in $A$. The \Cs{} $A_\omega \cap A'$ is called a \emph{central sequence algebra}. 

We shall here insist that the free filter $\omega$ is an ultrafilter. We avoid using the algebras $A^\infty := \ell^\infty (A)/c_0(A)$ and $A_\infty:=A'\cap A^\infty$, which have similar properties
and produce similar results (up to different selection procedures).
One reasons for preferring ultrafilters is that we have an epimorphism
from $A_\omega$ onto $N^\omega:= \ell_\infty(N)/c_{\tau,\omega}(N)$, cf.\ Theorem \ref{thm:M-omega} below, where $N$ is the weak closure of $A$ in the GNS representation determined by a tracial state $\tau$ on $A$. The algebra $N^\omega:= \ell_\infty(N)/c_{\tau,\omega}(N)$ is a $W^*$-algebra when $\omega$ is a free ultrafilter, whereas 
the sequence algebra $N^\infty:=\ell_\infty(N)/c_{\tau,0}(N)$
(with $c_{\tau,0}(N)$ the bounded sequences in $N$ with 
$\lim_n \| a\|_{2,\tau}=0$)
is not a $W^*$-algebra.
Another reason for preferring free ultrafilters to general free filters is that, for $A\not=\C$, simplicity of $A_\omega$ is equivalent to  pure infiniteness and simplicity of $A$.
The algebras $A_\omega$ are the fibres of the continuous field
$A^\infty$ with base space $\beta(\N)\setminus \N$. The structure of this bundle appears to be complicated.

Recall from \cite{Kir.AbelProc} that if $B$ is a \Cs{} and if $A$ is a separable sub-\Cs{} of $B_\omega$, then we define
the  relative central sequence algebra
$$F(A,B)=(A' \cap B_\omega)/\mathrm{Ann}(A,B_\omega).$$
This \Cs{} has interesting properties. For example, $F(A,B)=F(A\otimes \cK, B\otimes \cK)$. 
We let $F(A):=F(A,A)$. The \Cs{} $F(A)$ is unital, if $A$ is $\sigma$-unital. 
If $A$ is unital, then $F(A)=A'\cap A_\omega$. We refer to \cite{Kir.AbelProc} for a detailed account on the \Cs s $F(A,B)$ and $F(A)$. We shall often, in the unital case, denote the central sequence algebra $A_\omega \cap A'$ by $F(A)$. 

We have the following useful selection principle for  sequences in the filter $\omega$
with a countable number conditions from \cite[Lemma A.1]{Kir.AbelProc}.
For completeness we add a {proof:\footnote{With one misleading typo corrected!}}
\begin{lemma}[The $\ep$-test] \label{lem:omega.selection}
Let $\omega$ be a free ultrafilter. Let $X_1,X_2,\ldots$ be any sequence of sets
and suppose  that,
for each $k\in \N$, we are given a sequence
$(f^{(k)}_n)_{n \ge 1}$ of functions $f^{(k)}_n\colon\, X_n\to [0,\infty)$.

For each $k\in \N$ define a new function $f_\omega^{(k)} \colon \prod_{n=1}^\infty X_n \to [0,\infty]$ by 
$$f_\omega^{(k)}(s_1,s_2,\ldots)
=\lim_{n \to \omega}  f_n^{(k)}(s_n), \qquad (s_n)_{n \ge 1} \in \prod_{n=1}^\infty X_n.
$$
Suppose that, for each $m\in \N$  
and each $\varepsilon >0$, there exists $s=(s_1,s_2,\ldots)\in \prod_{n=1}^\infty X_n$ such that
$f_\omega^{(k)}(s)<\varepsilon$ for $k=1,2, \dots,m$. It follows that there is 
$t=(t_1,t_2,\ldots)\in \prod_{n=1}^\infty X_n$
with $ f_\omega^{(k)}(t)=0$ for all $k\in \N$.
\end{lemma}

\begin{proof} For each $n \in \N$ define a decreasing sequence $(X_{n,m})_{m \ge 0}$ of subsets of $X_n$ by $X_{n,0}=X_n$ and
$$
X_{n,m}=
\big\{ s\in X_n : \max\{f_n^{(1)}(s),\ldots,f_n^{(m)}(s)\}<1/m \big\},
$$
for $m \ge 1$. We let 
$m(n):=\max \{ m\leq n : X_{n,m}\ne \emptyset \}$; and for each integer $k \ge 1$, let  $Y_k := \{n\in \N : k\le m(n)\}$. 

Fix some $k \ge 1$. By assumption there exists $s = (s_n) \in \prod_n X_n$ such that $f_\omega^{(j)}(s) < 1/k$ for $1 \le j \le k$. This entails that the set
$$Z_k :=\big\{ n \in \N :  \max\{f_n^{(1)}(s_n),\ldots,f_n^{(k)}(s_n)\}<1/k \big\}$$
belongs to $\omega$ for each $k \ge 1$.  This again implies that $X_{n,k} \ne \emptyset$ for all $n \in Z_k$; which shows that $\min\{k,n\} \le m(n) \le n$ for all $n \in Z_k$. It follows that $Z_k\setminus \{1,2, \dots, k-1\} \subseteq Y_k$, from which we conclude that $Y_k \in \omega$ (because $\omega$ is assumed to be free). Now,
$$\lim_{n \to \omega} \; \frac{1}{m(n)} \; = \; \liminf_{n \to \omega}  \; \frac{1}{m(n)}  \; \le \;  
\inf_k \sup_{n \in Y_k}  \; \frac{1}{m(n)} \; \le \; \inf_k \; \frac{1}{k} = 0.$$
By definition of $m(n)$ we can find $t_n\in X_{n,m(n)}\subseteq X_n$ for each $n \in \N$. Put $t = (t_n)_{n \ge 1}$. Then $f_n^{(k)}(t_n)\leq 1/m(n)$ for all $k< m(n)$ by definition of $X_{n,m(n)}$, so 
$$ f_\omega^{(k)}(t) \; = \; \lim_{n \to \omega} f_n^{(k)}(t_n) \; \le \;  \lim_{n \to \omega} \; \frac{1}{m(n)} \; = \; 0,$$
for all $k \ge 1$ as desired.
\end{proof}

\noindent
Let $N$ be a $W^*$-algebra with separable predual and let  $\tau$ be a faithful
normal tracial state on $N$. Consider the associated norm, $\|a\|_{2,\tau} = \tau(a^*a)^{1/2}$, $a \in N$, on $N$. Let $N^\omega$ denote the $W^*$-algebra 
$\ell_\infty(N)/c_{\omega,\tau}(N)$, where 
$c_{\omega,\tau}(N)$
consists of the bounded sequences $(a_1,a_2,\cdots )$ with
$\lim_\omega \|a_n\|_{2,\tau}=0$. (As mentioned above, if $\omega$ is a free filter, which is not an ultrafilter, then $N^\omega$ is not a $W^*$-algebra.)

\begin{rem}\label{rem:central.sequences.in.N}
Since $N$ has separable predual, $N'\cap N^\omega$ contains
a copy of the hyper-finite $\mathrm{II}_1$-factor $\mathcal{R}$ with separable predual
if and only if for each $k\in \N$ there exists a sequence
of unital *-homomorphisms $\psi_n\colon M_k\to N$ such that
$\lim_{n\to \infty} \| [\psi_n(t),a] \|_{2,\tau}=0$ for all $t\in M_k$ and $a\in N$
if and only if there is a $^*$-homomorphism $\varphi\colon M_2 \to N'\cap N^\omega$
such that $t \in N\mapsto t\cdot \varphi(1)\in N^\omega$ is faithful.
(It suffices to consider elements $t$ in the center of $N$.) 

This equivalence
was shown by D.\ McDuff, \cite{McDuff.1970}, in the case where $N$ is a \emph{factor}. She gets, moreover, that $N\cong N\overline{\otimes} \mathcal{R}$   if and only if
 $N$ has a central sequence that is not hyper-central. Such a II$_1$-factor is called a \emph{McDuff factor}.
\end{rem}

\noindent
The result below was proved by Y.\ Sato in \cite[Lemma 2.1]{Sato.1104}  in the case where $A$ is nuclear. We give here an elementary proof of this useful result, that does not assume nuclearity of $A$. The result implies that the central sequence \Cs{} $A_\omega \cap A'$ has a subquotient isomorphic to the hyperfinite II$_1$-factor $\mathcal{R}$ whenever $A$ has a factorial trace so that the corresponding II$_1$-factor, arising from GNS representation with respect to that trace, is a McDuff factor. In Remark \ref{rem:J_tau} in the next section we show how one can give an easier proof of the theorem below using that the kernel $J_\tau$ of the natural \sh{} $A_\omega \to N^\omega$ is a so-called \emph{$\sigma$-ideal}. The proof given below does not (explicitly) use $\sigma$-ideals.

\begin{thm} \label{thm:M-omega}
Let $A$ be a separable unital \Cs, let $\tau$ be a faithful tracial state on $A$, let $N$ be the weak closure of $A$ under the GNS representation of $A$ with respect to the state $\tau$, and let $\omega$ be a free ultrafilter on $\N$. It follows that the natural \sh s
$$A_\omega \to N^\omega, \qquad A_\omega \cap A' \to N^\omega \cap N',$$
are surjective.
\end{thm}

\begin{proof}  Let $\pi_A$ and $\pi_N$ denote the quotient mappings $\ell^\infty(A) \to A_\omega$ and $\ell^\infty(N) \to N^\omega$, respectively. Denote the canonical map $A_\omega \to N^\omega$ by $\Phi$, and let $\widetilde{\Phi} \colon \ell^\infty(A) \to N_\omega$ denote the map $\Phi \circ \pi_A$. 

We show first that $\Phi \colon A_\omega \to N^\omega$ is onto. Indeed, if $x = \pi_N(x_1,x_2, \dots)$ is an element in $N^\omega$, then, by Kaplanski's density theorem, there exists $a_k \in A$ with $\|a_k\| \le  \|x_k\|$ and with $\|a_k-x_k\|_{2,\tau} \le 1/k$. It follows that $(a_1,a_2, \dots) \in \ell^\infty(A)$ and that  $\widetilde{\Phi}(a_1,a_2, \dots) = x$.

To prove that the natural map $A_\omega \cap A' \to N^\omega \cap N'$ is surjective, it suffices to show that if $b = (b_1,b_2, \dots) \in \ell^\infty(A)$ is such that $\widetilde{\Phi}(b) \in N^\omega \cap N'$, then there is an element $c \in \ell^\infty(A)$ such that $\pi_A(c) \in A_\omega \cap A'$ and $\widetilde{\Phi}(c) = \widetilde{\Phi}(b)$.  Let such a $b \in \ell^\infty(A)$ be given, put $B = C^*(A,b) \subseteq \ell^\infty(A)$, and put $J = \Ker(\widetilde{\Phi}) \cap B$. Notice that an element $x = (x_1,x_2 , \dots ) \in \ell^\infty(A)$ belongs to $\Ker(\widetilde{\Phi})$ if and only if $\lim_{n \to \omega} \|x_n\|_{2,\tau} = 0$. 
Notice also that $ba - a b\in J$ for all $a \in A$. 

Let $(d^{(k)})_{k \ge 1}$ be an increasing approximate unit for $J$, consisting of positive contractions, which is asymptotically central with respect to the separable \Cs{} $B$. Then
\begin{eqnarray*}
0 & = & \lim_{k \to \infty} \| (1-d^{(k)}) \big(ba-ab\big) (1-d^{(k)})\| \\
&=& \lim_{k \to \infty} \| (1-d^{(k)})b (1-d^{(k)})a- a(1-d^{(k)})b (1-d^{(k)})\|
\end{eqnarray*}
for all $a \in A$. 

We use the $\ep$-test (Lemma \ref{lem:omega.selection}) to complete the proof: Let $(a_k)_{k \ge 1}$  be a dense sequence in $A$. Let each $X_n$ be the set of positive contractions in $A$. Define $f_n^{(k)} \colon X_n \to [0,\infty)$ by
$$f_n^{(1)}(x) = \|x\|_{2,\tau}, \qquad f_n^{(k+1)}(x) = \|(1-x)b_n(1-x)a_k-a_k(1-x)b_n(1-x)\|, \; \; k \ge 1.$$
Notice that $f_\omega^{(1)}(d^{(\ell)}) = \lim_{n \to \omega}\|d_n^{(\ell)}\|_{2,\tau}= 0$ for all $\ell$, because each $d^{(\ell)}$ belongs to $J$. Note also that 
$$f_\omega^{(k+1)}(d^{(\ell)}) = \| (1-d^{(\ell)})b (1-d^{(\ell)})a_k- a_k(1-d^{(\ell)})b (1-d^{(\ell)})\|, \quad k \ge 1.$$
It is now easy to see that the conditions of the $\ep$-test in Lemma \ref{lem:omega.selection} are satisfied, so there exists a sequence $d=(d_n)_{n \ge 1}$ of positive contractions in $A$ such that  $f_\omega^{(k)}(d) = 0$ for all $k$. As $f_\omega^{(1)}(d) = 0$, we conclude that $d \in \Ker(\widetilde{\Phi})$. 

Put $c = (1-d)b(1-d)$. Then $c-b \in \Ker(\widetilde{\Phi})$, so $\widetilde{\Phi}(c) = \widetilde{\Phi}(b)$. Since $f_\omega^{(k+1)}(d) = 0$, we see that $ca_k-a_kc=0$ for all $k \ge 1$. This shows that $\pi_A(c) \in A_\omega \cap A'$. 
\end{proof}

\section{The trace-kernel ideal  $J_A$}
\label{sec:trace.kernel.ideal}

\noindent
Suppose that $A$ is a unital, separable \Cs{} that has at least one  tracial state.
Let $T(A)\subset A^*$ denote the Choquet simplex of tracial states on $A$. The topology on $T(A)$ will always be the weak$^*$ topology (or the $\sigma(A^*,A)$-topology) in which $T(A)$ is compact. 
The Choquet boundary $\partial T(A)$ (the set of extreme points in $T(A)$) is a Polish space, and it consists of factorial tracial states on $A$.

\begin{definition}[Seminorms on $A_\omega$] \label{def:seminorms}
For any non-empty subset $S$ of $T(A)$ and for any
$p\in [1,\infty)$, we define
a norm $\| \, \cdot \, \|_{p,S}$ on $A$ by 
$$\| a\|_{p,S}:= \sup_{\tau\in S}  \tau\big((a^*a)^{p/2}\big)^{1/p}=
\big[\sup_{\tau\in S}  \tau\big((a^*a)^{p/2}\big)\big]^{1/p}, \qquad a \in A.$$
We further define
$$\|a\|_p = \|a\|_{p,T(A)} = \sup_{\tau\in T(A)}  \tau\big((a^*a)^{p/2}\big)^{1/p}, \qquad a \in A.$$
If a sequence $\mathcal{S}=(S_1,S_2,\ldots )$ of subsets 
of $T(A)$  is given, 
then define seminorms on
$A_\omega$ by
$$\| a\|_{p,\mathcal{S}}:=\, 
\lim_{n \to \omega} \| a_n\|_{p, S_n}, \qquad \|  a  \|_{p,\omega}:=\| a\| _{p, \mathscr{T}},$$
for $a = \pi_\omega(a_1,a_2,a_3, \dots) \in A_\omega$; and where $\mathscr{T}:=(T(A),T(A),\ldots)$. 

If $S_n:=\{ \tau \}$ for all $n\in \N$, 
then we sometimes write  $\| a\|_{2,\tau}$ instead of $\| a\|_{2,\mathcal{S}}$, 
for $a\in A_\omega$.
\end{definition}

\noindent
For all $a \in A$  we have $\| a\|_{p,T}\leq \| a\|_{p,S} \leq \|a\|_{q,S}$
whenever $T\subseteq S \subseteq T(A)$ and $p\leq q$.
Moreover, $\| a\|_{p,T} = \| a\|_{p,S}$ if $T\subseteq S$ and
$S$ is contained in the closed convex hull of $T$. In particular, if $\partial T(A) \subseteq S$, then $\| a\|_{p,S} = \|a\|_p$. 

For each fixed non-empty $S \subseteq T(A)$, it is easy to see that $N(a):=\lim_{p\to\infty} \| a\|_{p,S}$ is a
semi-norm on $A$ with 
$$N(a)\leq \| a\|, \; \; N(ab)\leq N(a)\| b\|, \; \; N(a^*)=N(a), \; \; N( a)^2=N(a^*a)=N(aa^*).$$
Thus $\| a\|=\lim_{p\to\infty} \| a\|_{S,p}$ for $a\in A$, if $A$ is unital and simple. 

\smallskip

\begin{rem}\label{rem:partit.ultranorms}
Another equivalent description of $\| \, \cdot \, \|_{p,\omega}$ is given for $a:= \pi_\omega(a_1,a_2, \dots) \in A_\omega$ by the equation
$$\|a\|_{p,\omega} = \sup_{S \in M} \|a\|_{p,\mathcal{S}},$$
where $M$ denotes the set of all sequences $\mathcal{S} = (S_1,S_2, \dots)$, where $S_n =\{\tau_n\}$ is a singleton for all $n$, i.e., $\|a\|_{p, \mathcal{S}} = \lim_{n \to \omega} \|a_n\|_{p,\tau_n}$. Again it suffices to use sequences $(\tau_n)_{n \ge 1}$ in $\partial T(A)$. 
\end{rem}

\medskip

\begin{definition}[The trace-kernel ideal $J_A$]
Define a linear subspace $J_A$ of $A_\omega$ by
$$J_A = \{ e\in A_\omega : \| e \|_{2,\omega}=0\}.$$
Since $\| e^*\| _{2,\omega}=\| e \|_{2,\omega}$, 
$\| fe \| _{2,\omega}\leq 
\| f\| \cdot \| e \|_{2,\omega}$ and
$\| e \|_{2,\omega}\leq \| e \|$ for all $e,f \in A_\omega$
the subspace  $J_A$ is a closed
two-sided ideal of $A_\omega$.
\end{definition}

\noindent
We can in the definition of $J_A$ replace the semi-norm
 $\| \, \cdot \, \|_{2,\omega}$ by any of the semi-norms   
 $\| \, \cdot \, \|_{p,\omega}$,
 with $p\in [1,\infty)$. Indeed, for each $a \in A$, we have the inequalities
$\|a\|_{1,\tau}  \leq  \| a \|_{p,\tau}$;
and 
$$\| a \|_{p,\tau} = \Big(\| (a^*a)^{p/2} \| _{1,\tau}\Big)^{1/p}\leq  
(\| a\| _{1,\tau})^{1/p} \cdot \| a\|^{1-1/p}.$$ 
Hence the following holds for all contractions $e\in A_\omega$:
\begin{equation} \label{eq:seminorms}
\| e \|_{1,\omega}\leq \| e \|_{p,\omega}\, 
\leq \| e \|_{1,\omega}^{1/p}.
\end{equation}

\medskip
\noindent We recall some facts about $\sigma$-ideals from \cite{Kir.AbelProc}. 
The definitions below are also given in \cite[Definition 1.4 and 1.5]{Kir.AbelProc}.
\begin{definition}
\label{def:sigma.ideals.etc}
Let $B$ a \Cs{} and let $J$ be a closed two-sided ideal in $B$.
Then $J$  is  a \emph{$\sigma$-ideal} of $B$ 
if for every separable sub-\Cs{} $C$ of $B$
there exists a positive contraction $e\in J$
such that $e\in C'\cap J$ and $ec=c$ for all $c\in C\cap J$.

A short exact sequence  
$$\xymatrix{0\ar[r] & J \ar[r] &  B \ar[r]^{\pi_J} & B/J \ar[r] & 0}$$
is \emph{strongly locally semi-split},
if for every
separable sub-\Cs{} $D$ of $B/J$ there
exists a  \sh{} $\psi\colon \Cf_0((0,1],D)\to B$
with $\pi_J(\psi(\iota \otimes d))=d$ for all $d\in D$,
where $\iota \in \Cf_0((0,1])$ is given by $\iota(t)=t$, for $t\in (0,1]$.

A \Cs{} $B$  is \emph{$\sigma$-sub-Stonean} 
if for every separable sub-\Cs{} $D$ of $B$
and for every positive contractions $b,c\in B$ with
$bDc=\{ 0\}$ and $bc=0$, there exist positive contractions
$e,f\in D'\cap B$ with $eb=b$, $fc=c$, and $ef=0$.
\end{definition}

\noindent
We list some properties of $\sigma$-ideals.
\begin{prop}\label{prop:sigma.ideals.etc}
Suppose that $B$ is a \Cs, that $J$ is a $\sigma$-ideal in $B$, and that 
$C$ is a separable sub-\Cs{} of $B$.
Then $J$, $B$, and $C\,'\cap B$  have the following properties:
\begin{itemize}
\item[(i)]
$J\cap (C\,' \cap B)$ is a $\sigma$-ideal of
$C\,'\cap B$, and $C\,'\cap J$
contains an approximate
unit of $J$. \vspace{.1cm}
\item[(ii)]
If $D$ is a separable \Cs{} and 
$\varphi\colon D\to \pi_J(C)\,' \cap (B/J)$ is a \sh,
then there exists a \sh{}
$\psi\colon \Cf_0((0,1],D)\to C\,' \cap B$
with $\pi_J(\psi(\iota \otimes d))=\varphi(d)$. \vspace{.1cm}
\item[(iii)] $\pi_J(C\,' \cap B)=\pi_J(C)\,' \cap (B/J)\,\,$, and
$$
0\to\, J\cap (C\,' \cap B) \,\to\, C\,' \cap B \,\to 
(C\,' \cap B)/\bigl(J\cap (C\,' \cap B)\bigr)
\to 0
$$
is strongly locally semi-split. \vspace{.1cm}
\item[(iv)]
If  $B$ is $\sigma$-sub-Stonean,
then
$B/J$ is $\sigma$-sub-Stonean.
\end{itemize}
\end{prop}
\begin{proof} Part (i) follows immediately from the definition.
Parts (ii) and (iii) follow from
 \cite[prop.1.6, proof in appendix B, p.221]{Kir.AbelProc}.
See \cite[prop.1.3, proof in appendix B, p.219]{Kir.AbelProc}
for (iv).
\end{proof}

\begin{prop}\label{rem:local.units.in.JA} 
Let $A$ be a unital \Cs{} with $T(A) \ne \emptyset$, and let $\omega$ be a free filter. It follows that the trace-kernel ideal $J_A$ is a $\sigma$-ideal in $A_\omega$,
and that $A_\omega$ and $F(A)$ are $\sigma$-sub-Stonean \Cs s. If $A$ is separable, then $J_A \cap F(A)$ is a $\sigma$-ideal in $F(A)$.
\end{prop}
\begin{proof} The fact that $A_\omega$ and $F(A)$  are  
$\sigma$-sub-Stonean was shown in \cite[Proposition 1.3]{Kir.AbelProc}.

We show below that $J_A$ is a $\sigma$-ideal in $A_\omega$. It will then follow from Proposition \ref{prop:sigma.ideals.etc} (i) that $J_A \cap F(A)$ is a $\sigma$-ideal in $F(A)$, if $A$ is separable. 

Let $C$ be a separable sub-\Cs{} of $A_\omega$ and let $d$ be a strictly positive contraction in the separable \Cs{} $J_A \cap C$.  It suffices to find a positive contraction $e \in J_A \cap C'$ such that $ed=d$.

Take contractions  $(c^{(k)}_1,c^{(k)}_2,\ldots)\in  \ell^\infty(A)$ such that $c^{(k)}:= \pi_\omega (c^{(k)}_1,c^{(k)}_2,\ldots)$ is a
dense sequence in the unit ball of $C$. Let
$(d_1,d_2, \ldots)\in \ell^\infty(A)$ be a representative of $d$. Let $(e^{(k)})_{k \ge 1}$ be an increasing approximate unit for $J_A \cap C$, consisting of positive contractions, which is asymptotically central for $C$. Let $(e_1^{(k)}, e_2^{(k)}, e_3^{(k)}, \dots) \in \ell^\infty(A)$ be a representative for $e^{(k)}$. 

We use Lemma \ref{lem:omega.selection} to show that there is $e$ with the desired properties. Let each $X_n $ be the set of positive contractions in $A$, and define functions $f_n^{(k)} \colon X_n \to [0,\infty)$ by
$$
f^{(1)}_n(x)=  \|(1-x)d_n \|, \qquad
f^{(2)}_n(x)= \| x\|_2, \qquad
f^{(k+2)}_n (x)=\| c^{(k)}_n x - xc^{(k)}_n \|, \quad k \ge 1.
$$
Note that $f^{(2)}_\omega\big((e_1^{(\ell)},e_2^{(\ell)},e_3^{(\ell)}, \dots)\big)= \| e^{(\ell)}\|_{2, \omega} = 0$ for all $\ell$, and that 
\begin{eqnarray*}
 \lim_{\ell \to \infty} \; f^{(1)}_\omega\big((e_1^{(\ell)},e_2^{(\ell)},e_3^{(\ell)}, \dots)\big) 
&=& \lim_{\ell \to \infty} \; \|(1-e^{(\ell)})d\| \;= \;0, \\
 \lim_{\ell \to \infty} \; f^{(k+2)}_\omega \big((e_1^{(\ell)},e_2^{(\ell)},e_3^{(\ell)}, \dots)\big) 
&=&\lim_{\ell \to \infty} \;  \| c^{(k)} e^{(\ell)}- e^{(\ell)}c^{(k)} \| \;= \;0.
\end{eqnarray*}
Lemma \ref{lem:omega.selection} therefore shows that there is a sequence, $(e_n)_{n \ge 1}$, of positive contractions in $A$ such that
\begin{equation} \label{eq:zero}
f_\omega^{(k)}(e_1, e_2, e_3, \dots) = 0
\end{equation}
for all $k$. The element $e = \pi_\omega(e_1, e_2, e_3, \dots)$ therefore belongs to $J_A$ (because \eqref{eq:zero} holds for $k=2$), it commutes with all elements $c^{(k)}$  (because \eqref{eq:zero} holds for all $k \ge 3$), and hence $e$ commutes with all elements in $C$; and $ed=d$  (because \eqref{eq:zero} holds for $k=1$).
\end{proof}

\begin{rem} \label{rem:J_tau}
Let $A$ be a unital \Cs{} with $T(A) \ne \emptyset$. For any family $\mathcal{S} = (S_n)_{n\ge 1} $ of non-empty subsets of $T(A)$ one can define the ideal $J_{\mathcal{S}}$ of $A_\omega$ to be the closed two-sided ideal of those $a \in A_\omega$ for which $\|a\|_{2,\mathcal{S}} = 0$. By almost the same argument as in the proof of the previous proposition one can show that $J_{\mathcal{S}}$ is a $\sigma$-ideal in $A_\omega$. (One need only replace $f_n^{(2)}$ with $f_n^{(2)}(x) = \|x\|_{2,S_n}$.)

In particular, $J_\tau$ is a $\sigma$-ideal in $A_\omega$ for each $\tau \in T(A)$ (and where $J_\tau = J_{\mathcal{S}}$ with  $S_n =\{\tau\}$ for all $n$). In other words, $a \in J_\tau$ if and only if $\|a\|_{2,\tau} = 0$. 

One can use the fact that $J_\tau$ is a $\sigma$-ideal in $A_\omega$ to give a shorter proof of Theorem \ref{thm:M-omega}: Follow the first part of the proof of  Theorem \ref{thm:M-omega} to the place where it is shown that the natural map 
$\Phi \colon A_\omega  \to N^\omega$
 is surjective. Note the the kernel of $\Phi$ is equal to $J_\tau$.
We must show that if $b \in A_\omega$ is such that $\Phi(b) \in  N^\omega \cap N'$, then there exists $c \in A_\omega \cap A'$ such that $\Phi(c)  = \Phi(b)$. Let $B$ be the separable sub-\Cs{} of $A_\omega$ generated by $A$ and the element $b$. Then, by the fact that $J_\tau$ is a $\sigma$-ideal in $A_\omega$, there is a positive contraction $e \in J_\tau \cap B'$ such that $ex=x$ for all $x \in J_A \cap C$. Put $c = (1-e)b(1-e)$. Then $c-b \in J_\tau$, so $\Phi(c) = \Phi(b)$. For each $a \in A$ we have $ab-ba \in \mathrm{Ker}(\Phi) = J_\tau$,  because $\Phi(b) \in N^\omega \cap N'$. Hence, $ab-ba \in J_\tau \cap C$, so 
$$0 = (1-e)(ab-ba)(1-e) =a(1-e)b(1-e)-(1-e)b(1-e)a = ac-ca.$$
This shows that $c \in A_\omega \cap A'$.
\end{rem}


\section{Excision in small central sequences and property (SI)}
\label{sec.excision}

\noindent In this section we show that "excision in small central sequences" (cf.\ Definition \ref{def:omega.related.excision}) holds under the minimal assumptions that the ambient \Cs{} $A$ is unital and has the local weak comparison property (cf.\ Definition \ref{def:weak.comparison}) and that the completely positive map (to be excised) is nuclear. We do not assume that $A$ is nuclear. The main forte of our approach is that we do not use \cite[Lemma 3.3]{Matui.Sato} (or other parts Section 3 of \cite{Matui.Sato}); and that we hence obtain "excision in small central sequences"  property using less machinery and under fewer assumptions. The crucial property (SI) follows from "excision in small central sequences"  applied to the identity map as in Matui--Sato, \cite{Matui.Sato}.  The identity map is nuclear if and only if the \Cs{} is nuclear, so our methods give property (SI) only for nuclear \Cs s, but under a weaker comparability condition. 

We remind the reader of the original definition of property (SI) from 
\cite[Definition~4.1]{Matui.Sato.CMP}.

\begin{definition}[Matui--Sato]
\label{def:Matui.Sato.SI}
A separable \Cs{} $A$ with $T(A) \ne \emptyset$ is said to have property (SI) if, for any central sequences 
$(e_n)_{n \ge 1}$ and $(f_n)_{n \ge 1}$ of positive contractions in $A$ satisfying
\begin{equation} \label{eq:SI1}
\lim_{n\to\infty} \max_{\tau\in T(A)} \tau(e_n) = 0, \qquad 
\lim_{m\to \infty} \liminf_{n\to \infty}  \min_{\tau\in T(A)} \tau(f_n^m) > 0,
\end{equation}
there exists a central sequence $(s_n)_{n \ge 1}$ in $A$ such that
\begin{equation} \label{eq:SI2}
\lim_{n \to \infty} \|s_n^*s_n - e_n\| = 0, \qquad \lim_{n\to \infty} \| f_ns_n - s_n\|=0.
\end{equation} 
\end{definition}

\medskip
\noindent The definition above  is equivalent with our Definition \ref{def:property(SI)} (expressed in terms of the central sequence algebra), as well as to the "local" definition of (SI) in part (ii) of the lemma below.   Notice that 
\begin{equation} \label{eq.trace.estimates}
\max_{\tau \in T(A)} \tau(e) = \|e\|_1, \qquad \min_{\tau \in T(A)} \tau(f) = 1-\|1-f\|_1,
\end{equation}
for all positive elements $e \in A$, and for all positive contractions $f \in A$. 
\smallskip

\begin{lemma} \label{lem:SI-eq}
The following conditions are equivalent for any unital, simple, separable \Cs{} $A$ with $T(A) \ne \emptyset$:
\begin{itemize}
\item[(i)] $A$ has property (SI), in the sense of Definition \ref{def:Matui.Sato.SI}
\item[(ii)] For all finite subsets $F \subset A$, for all $\ep > 0$, and for all $0 < \rho < 1$, there exist $\delta > 0$, a finite subset $G \subset A$, and $N \in \N$, such that for all positive contractions $e,f \in A$ with
$$
\max_{a \in G} \big(\|[e,a]\| + \| [f,a]\|\big) < \delta, 
\qquad \|e\|_1 <  \delta, \qquad \max_{1 \le k \le N} \|1-f^k\|_1 \le 1- \rho,
$$
there exists $s \in A$ such that
$$ \max_{a \in F} \|[s,a]\| < \ep, \qquad \|(1-f)s\| < \ep, \qquad  \|s^*s-e\| < \ep.$$
\item[(iii)] $A$ has property (SI) in the sense of Definition \ref{def:property(SI)} for every/some free ultrafilter. 
\end{itemize}
\end{lemma}

\noindent The proof of the implication of (ii) $\Rightarrow$ (iii) below is valid for every ultrafilter $\omega$, whereas the proof of the implication of (iii) $\Rightarrow$ (i) only requires (iii) to hold for some ultrafilter.

\begin{proof} (i) $\Rightarrow$ (ii). Suppose that (ii) does not hold. Let $(G_n)_{n \ge 1}$ be an increasing sequence of finite subsets of $A$ whose union is dense in $A$.  The negation of (ii) implies that  there exists a finite subset $F \subset A$, and there exist $\ep > 0$ and $0 < \rho < 1$, such that, for $\delta = 1/n$, $G=G_n$ and $N=n$ (where $n$ is any natural number), there exist positive contractions $e_n,f_n \in A$ satisfying
$$
\max _{a \in G_n} \big(\|[e_n,a]\| + \| [f_n,a]\|\big) < 1/n, 
\qquad \|e_n\|_1 <  1/n, \qquad \max_{1 \le k \le n} \|1-f_n^k\|_1 \le 1- \rho,
$$
while there is no $s \in A$ such that 
\begin{equation} \label{eq:SI-cond}
\max_{a \in F} \|[s,a]\| < \ep, \qquad \|(1-f_n)s\| < \ep, \qquad  \|s^*s-e_n\| < \ep.
\end{equation}
The sequences $(e_n)_{n \ge 1}$ and $(f_n)_{n \ge 1}$ will satisfy the conditions in  \eqref{eq:SI1}. However, if $(s_n)_{n \ge 1}$ were a central sequence satisfying  \eqref{eq:SI2}, then $s_n$ would satisfy \eqref{eq:SI-cond} for some $n$. Hence (i) does not hold.

(ii) $\Rightarrow$ (iii). Let $\omega$ be any free ultrafilter on $\N$, and suppose that (ii) holds. Recall that $F(A) = A_\omega \cap A'$ (the ultrafilter $\omega$ is suppressed in our notation\footnote{This is justified, because $A_\omega \cong A_{\omega'}$ for any two ultrafilters $\omega$ and $\omega'$, if $A$ is separable and if the continuum hypothesis holds. Moreover, without assuming the continuum hypothesis, but assuming the axiom of choice, which is needed for the existence of free ultrafilters, the isomorphism classes of separable unital sub-\Cs s of $A_\omega$ (and of $F(A)$) are the same for all free ultrafilters.}). Let $e,f \in F(A)$ be positive contractions such that $e \in J_A$ and $\sup_k \|1-f^k\|_{1,\omega} < 1$; and let $(e_n)_{n \ge 1}$ and $(f_n)_{n \ge 1}$ be positive contractive lifts in $\ell^\infty(A)$ of $e$ and $f$. It follows that 
$$\rho_0 := \sup_k \lim_{n \to \omega}\inf_{\tau \in T(A)} \tau(f_n^k) = 1- \sup_k \|1-f^k\|_{1,\omega} > 0,$$
cf.\ \eqref{eq.trace.estimates}. Let $(a_k)_{k \ge 1}$ be a dense sequence in $A$. We use Lemma \ref{lem:omega.selection} to find $s \in F(A)$ such that $fs=s$ and $s^*s = e$. This will then show that $A$ has property (SI) in the sense of Definition \ref{def:property(SI)}. Let each $X_n$ be the set of all elements in $A$ of norm at most $2$, and let $f^{(k)}_n \colon X_n \to [0,\infty)$ be defined by 
$$f^{(1)}_n(x) = \|(1-f_n)x\|, \qquad f^{(2)}_n(x) = \|x^*x-e_n\|, \qquad f^{(k+2)}_n(x) = \|[x,a_k]\|, \; k \ge 1.$$
We wish to find $\bar{s}=(s_1,s_2,s_3, \dots) \in \ell^\infty(A)$ such that $f^{(k)}_\omega(\bar{s}) = 0$ for all $k \ge 1$. It will then follow that $s = \pi_\omega(\bar{s})$ satisfies $s \in F(A)$, $fs = s$ and $s^*s=e$, and we are done. By Lemma \ref{lem:omega.selection} it suffices to show that, for each $\ep > 0$ and each integer $m \ge 1$, there exists $\bar{s} \in \ell^\infty(A)$ such that $f^{(k)}_\omega(\bar{s}) \le \ep$ for $1 \le k \le m+2$. Fix such $\ep >0$ and $m \ge 1$.

Let $\delta > 0$, $G \subset A$, and $N \in \N$ be associated, as prescribed in (ii), to the triple consisting of $F=\{a_1,a_2, \dots, a_m\}$, $\ep$ (as above), and $\rho=\rho_0/2$. Let $X$ be the set of all $n \in \N$ such that 
$$\max_{a \in G} \big( \|[e_n,a]\| + \|[f_n,a]\|\big)< \delta, \qquad \|e_n\|_1 < \delta \qquad \max_{1 \le k \le n} \|1-f^k\|_1 \le 1-\rho.$$ 
Then $X \in \omega$. The conclusion of (ii) then states that, for each $n \in X$, there is $s_n \in A$ such that 
$$\max_{1 \le j \le m} \|[s_n,a_j]\| < \ep, \qquad \|(1-f_n)s_n\| < \ep, \qquad \|s_n^*s_n - e_n\| < \ep.$$ Put $s_n=0$ if $n \notin X$. Then $\bar{s} = (s_1,s_2,s_3, \dots)$ satisfies $f^{(k)}_\omega(\bar{s}) \le \ep$ for $1 \le k \le m+2$.

(iii) $\Rightarrow$ (i). Suppose that (i) does not hold. Then there exist central sequences $(e_n)_{n \ge 1}$ and $(f_n)_{n \ge 1}$ of positive contractions in $A$ satisfying \eqref{eq:SI1}, but such that there is no sequence $(s_n)_{n \ge 1}$  in $A$ which satisfies \eqref{eq:SI2}. This means that there exists $\delta > 0$, a finite subset $F \subset A$, and a sub-sequence $(n_\ell)_{\ell \ge 1}$, such that 
\begin{equation} \label{eq:obstruct1}
\max\big\{\|s^*s-e_{n_\ell}\|, \; \|(1-f_{n_\ell})s\|, \, \max_{a \in F} \|[s,a]\| \big\} \ge \delta,
\end{equation}
for all $\ell \ge 1$ and all $s \in A$. Put
$$e = \pi_\omega(e_{n_1},e_{n_2}, e_{n_3}, \dots), \qquad f = \pi_\omega(f_{n_1},f_{n_2}, f_{n_3}, \dots).$$
Then $e,f \in F(A)$, $e \in J_A$ and 
$$\sup_k \|1-f^k\|_{1,\omega} = \sup_k \, \lim_{\ell \to \omega} \sup_{\tau \in T(A)} \tau(1-f_{n_\ell}^k) < 1.$$
However, the existence of  $s =\pi_\omega(s_1,s_2, \dots) \in F(A)$, satisfying $s^*s = e$ and $fs = s$, would imply that the set of $\ell \in \N$ for which 
$$\|s_\ell^* s_\ell - e_{n_\ell}\| < \delta, \qquad \|(1-f_{n_\ell})s_\ell\|< \delta , \qquad \max_{a \in F} \|[s_\ell,a]\| < \delta,$$
belongs to $\omega$, and hence is non-empty.  This contradicts \eqref{eq:obstruct1}, so (iii) cannot hold. 
\end{proof}

\noindent
We also remind the reader of the original definition of excision  in small central sequences from 
\cite[Definition 2.1]{Matui.Sato}.

\begin{definition}[Matui--Sato]
\label{def:Matui.Sato.excision}
Let $A$ be a separable \Cs{} with $T(A) \ne \emptyset$. A completely positive map 
$\varphi\colon  A \to  A$ can be
\emph{excised in small central sequences} if, for any central sequences 
$(e_n)_{n \ge 1}$ and $(f_n)_{n \ge 1}$ of positive contractions in $A$ satisfying
\begin{equation} \label{eq:excision1}
\lim_{n\to\infty} \max_{\tau\in T(A)} \tau(e_n) = 0, \qquad 
\lim_{m\to \infty} \liminf_{n\to \infty}  \min_{\tau\in T(A)} \tau(f_n^m) > 0,
\end{equation}
there exists a sequence $(s_n)_{n \ge 1}$ in $A$ such that
\begin{equation} \label{eq:excision2}
\lim_{n\to \infty} \| f_ns_n - s_n\|=0, \qquad \lim_{n\to \infty} \| s_n^* a s_n - \varphi(a)e_n\|=0,
\end{equation} 
for all $a \in A$.
\end{definition}

\noindent Note that the sequence $(s_n)_{n \ge 1}$ in Definition \ref{def:Matui.Sato.excision} automatically satisfies $\|s_n\| \to 1$ if $A$ is unital. Hence we can assume, without loss of generality, that all $s_n$ in  Definition \ref{def:Matui.Sato.excision} are contractions. The same is true when $A$ is non-unital: Take an increasing approximate unit $(a_n)_{n \ge 1}$ for $A$, consisting of positive contractions, and replace the sequence $(s_n)_{n \ge 1}$ with $(a_n^{1/2}s_n)_{n \ge 1}$ (possibly after passing to suitable subsequences of both $(a_n)$ and $(s_n)$). Then, again, we get $\|s_n\| \to 1$, and we can resize and assume that $\|s_n\| = 1$ for all $n$. 

In a similar way we see that the element $s \in A_\omega$ from Definition \ref{def:omega.related.excision}  can be chosen to be a contraction. (If $A$ is unital, then $s$ is automatically an isometry.)


\begin{lemma} \label{lem:excision-eq}
The following conditions are equivalent for any unital separable \Cs{} $A$ with $T(A) \ne \emptyset$, and any completely positive map $\varphi \colon A \to A$:
\begin{itemize}
\item[(i)] $\varphi$ can be excised in small central sequences in the sense of Definition \ref{def:Matui.Sato.excision}.
\item[(ii)] For all finite subsets $F \subset A$, for all $\ep > 0$, and for all $0 < \rho < 1$, there exist $\delta > 0$, a finite subset $G \subset A$, and $N \in \N$, such that for all positive contractions $e,f \in A$ with
$$
\max_{a \in G} \big(\|[e,a]\| + \| [f,a]\|\big) < \delta, 
\qquad \|e\|_1 <  \delta, \qquad \sup_{1 \le k \le N} \|1-f^k\|_1 \le 1- \rho,
$$
there exists  $s \in A$ such that
$$\|(1-f)s\| < \ep, \qquad \max_{a \in F} \|s^*as - \varphi(a)e_n\| < \ep.$$
\item[(iii)] $\varphi$ can be excised in small central sequences in the sense of Definition \ref{def:omega.related.excision} for every/some free ultrafilter $\omega$.
\end{itemize}
\end{lemma}

\noindent The proof of Lemma \ref{lem:excision-eq} is almost identical to the one of Lemma \ref{lem:SI-eq}, so we omit it.

Part (ii) of the lemma below shows that \cite[Lemma 2.5]{Matui.Sato} holds under the weaker assumption that $A$ has the local weak comparison property, cf.\ Definition \ref{def:weak.comparison}, (rather than strict comparison). Part (i) of the lemma is a translation of \cite[Lemma 2.5]{Matui.Sato} into central sequence algebra language.

The proof of the lemma uses the following fact: Suppose that $\psi \colon A \to B$ is an epimorphism between \Cs s $A$ and $B$, and that $e,f \in B$ are positive contractions such that $ef=f$. Then there exists positive contractions $e',f' \in A$ such that $\psi(e') = e$, $\psi(f') = f$, and $e'f'=f'$. This can be seen by using the well-known fact that pairwise orthogonal positive elements lift to pairwise orthogonal positive elements (with the same norm) applied to $f'$ and $1-e'$ in (the unitization of) $B$.

\begin{lemma}\label{lem:from.weak.comparison.to.estimate}
Suppose that $A$ is a unital, separable \Cs{} with $T(A) \ne \emptyset$. 
Let  $e, f\in A_\omega$ be positive contractions with 
$e \in J_A$,
$\sup_k \| 1-f^k \|_{1,\omega} <1$, and $f\in F(A)$.
\begin{itemize}
\item[(i)] There are
$e_0, f_0\in F(A)$ with  $e_0 \in J_A$, 
$$ e_0e=e, \qquad  f_0f=f_0, \qquad \|1-f_0\|_{1,\omega} = 
\sup_k  \| 1-f_0^k \|_{1,\omega}= \sup_k \| 1-f^k \|_{1,\omega}.
$$
\item[(ii)]
If $A$, in addition, is simple and has the local weak comparison property, and $QT(A)=T(A)$, then
for every non-zero positive element $a$ in $A$ and for every $\ep > 0$ there exists
$t \in A_\omega$ with 
$$ t^*a t=e,  \qquad ft =t, \qquad \|t\| \le \|a\|^{1/2} + \ep.$$
\end{itemize}
\end{lemma}

\begin{proof} (i). The existence of $e_0 \in F(A)$ with $e_0 \in J_A$ and $e_0e=e$ follows immediately from the fact that $J_A$ is a $\sigma$-ideal, cf.\ Proposition \ref{rem:local.units.in.JA} and Definition \ref{def:sigma.ideals.etc}. (Note that we do not need to assume that $e \in F(A)$ to obtain this.)

We use Lemma \ref{lem:omega.selection} to prove the existence of the element $f_0$. Lift $f$ to a positive contraction $(f_1,f_2,f_3, \dots)$ in $\ell^\infty(A)$. Put $\rho = \sup_k \|1-f^k\|_{1,\omega}$, and let
 $(a_k)_{k \ge 1}$ be a dense sequence in $A$. Let each of the sets $X_n$ of 
Lemma \ref{lem:omega.selection} be the set of all positive contractions in $A$, and define the functions $g_n^{(k)} \colon X_n \to [0,\infty)$ by
$$\vspace{.1cm} g_n^{(1)}(x) = \|x(1-f_n)\|, \quad g_n^{(2k)}(x) = \max\{\|1-x^k \|_{1,\omega}-\rho,0\}, \quad 
g_n^{(2k+1)}(x)  = \|xa_k-a_kx\|,\vspace{.1cm}
$$
for $k \ge 1$ and $x \in X_n$. Fix $\ell \in \N$, and put $s_n = (f_n)^\ell \in X_n$. Then 
\begin{eqnarray*}
g_\omega^{(1)}\big(s_1,s_2,s_3, \dots \big) &= &\|f^\ell(1-f)\| \;\; \overset{\ell \to \infty}{\longrightarrow} \; \; 0,\\
g_\omega^{(2k)}\big(s_1,s_2,s_3, \dots\big)  &= &\max\big\{\|1-f^{\ell \cdot k}\|_{1,\omega}-\rho,0\big\} 
\; =\; 0,\\
g_\omega^{(2k+1)}\big(s_1,s_2,s_3, \dots \big)  &= & \|f^\ell a_k - a_k f^\ell \| \; =\; 0.
\end{eqnarray*}
 Lemma~\ref{lem:omega.selection}  now gives the existence of a sequence $(f_{0,n})_{n\ge1}$ of positive contractions in $A$ such that $g_\omega^{(k)}(f_{0,1},f_{0,2}, f_{0,3}, \dots) = 0$ for all $k$. The positive contraction $f_0 = \pi_\omega(f_{0,1},f_{0,2}, f_{0,3}, \dots) \in A_\omega$ then has the following properties: $f_0f=f_0$, $\|1-f_0^\ell \|_{1,\omega} \le \sup_k \|1-f^k\|_{1,\omega}$ for all $\ell \ge 1$, and $f_0 \in F(A)$ (i.e., $f_0$ commutes with all elements of $A$). From the first identity we conclude that $f^kf_0 = f_0$ for all $k$, whence $1-f^k \le 1-f_0$ for all $k$, which again implies that $\sup_k \|1-f^k\|_{1,\omega} \le \|1-f_0\|_{1,\omega}$.

(ii). We can without loss of generality assume that $\|a\| = 1$. By the continuous function calculus we can find positive elements $g,h \in C^*(1,a)$ such that $\|g\| = 1$, $\|h\| \le 1 + \ep$, and $ahg=g$. 

Let $e_0$ and $f_0$ be as in part (i), and let $(e_{0,n})_{n \ge 1}, (f_{0,n})_{n \ge 1} \in \ell^\infty(A)$ be positive contractive lifts of $e_0$ and $f_0$, respectively. As remarked above the lemma we can choose the lift of $e_0$ such that $e_{0,n}e_n =e_n$ for all $n$. Put
$$\eta := \lim_{n \to \omega} \,  \inf_{\tau \in T(A)} \tau(f_{0,n}) = 1-\|1-f_0\|_{1,\omega} > 0.$$
By the proof of \cite[Lemma 2.4]{Matui.Sato} there is a constant $\alpha > 0$ (that only depends on $g$) such that
$$\lim_{n \to \omega} \tau(f_{0,n}^{1/2}gf_{0,n}^{1/2}) \ge \alpha \cdot \lim_{n\to\omega} \tau(f_{0,n})$$
for all $\tau \in T(A)$. Put $\delta = \alpha \eta/2 > 0$ and put $b_n = (f_{0,n}^{1/2}gf_{0,n}^{1/2} - \delta)_+$. As $\tau(b) \le d_\tau(b)$ for all positive contractions $b \in A$ and all $\tau \in T(A)$, it follows that
\begin{eqnarray*}
\lim_{n \to \omega} \,  \inf_{\tau \in T(A)} d_\tau(b_n)  & \ge & \lim_{n\to \omega} \, \inf_{\tau \in T(A)} \tau(b_n) \; \ge \; 
\lim_{n\to \omega} \, \inf_{\tau \in T(A)} \tau(f_{0,n}^{1/2}gf_{0,n}^{1/2})  - \delta\\
&\ge& \alpha \cdot \lim_{n \to\omega} \, \inf_{\tau \in T(A)} \tau(f_{0,n}) - \delta \; = \; \alpha  \eta - \delta \; > \; 0.
\end{eqnarray*}

We claim that
\begin{equation} \label{eq:e_0}
\lim_{n \to \omega} \, \sup_{\tau \in T(A)} d_\tau(e_{0,n}) = 0.
\end{equation}
To see this, apply the  result about the existence of $e_0$ in part (i) once again to  find a positive contraction $e' \in F(A) \cap J_A$ such that $e'e_0 = e_0$. Let $(e'_n)_{n \ge 1}$  in $\ell^\infty(A)$ be a positive contractive lift of $e'$ such that $e'_ne_{0,n} = e_{0,n}$  for all $n$. Then $d_\tau(e_{0,n}) \le \tau(e'_n)$ for all $n$ and all $\tau \in T(A)$, and $\lim_{n \to \omega} \sup_{\tau \in T(A)} \tau(e'_n) = \|e'\|_{1,\omega} =0$. This shows that \eqref{eq:e_0} holds. 

Let $\gamma = \gamma(A)$ be the constant witnessing that $A$ has local weak comparison. The set $X$ consisting of all $n \in \N$ such that
$$\gamma \cdot \sup_{\tau \in T(A)} d_\tau(e_{0,n}) < \inf_{\tau \in T(A)} d_\tau(b_n)$$
belongs to $\omega$. By definition of $\gamma$ this entails that $e_{0,n} \precsim b_n = (f_{0,n}^{1/2}gf_{0,n}^{1/2}-\delta)_+$ for all $n \in X$. (Here we use the assumption that $T(A) = QT(A)$.)  As $e_{0,n}e_n = e_n$, this implies we can find $v_n \in A$ such that $v_n^*f_{0,n}^{1/2}gf_{0,n}^{1/2}v_n = e_n$ and $\|v_n\| \le \delta^{-1/2}$ for all $n \in X$. Observe that $\|g^{1/2}f_{0,n}^{1/2}v_n\|^2 = \|e_n\| \le 1$. Put  $t_n = h^{1/2}g^{1/2}f_{0,n}^{1/2}v_n$. Then 
$$\|t_n\|^2 \le 1+\ep, \quad t_n^*at_n = e_n, \quad \|(1-f_n)t_n\|  \le \big\|\big[f_n,h^{1/2}g^{1/2}\big]\big\| \, \|f_{0,n}v_n\|.
$$
Hence $\lim_\omega \|(1-f_n)t_n\|  =0$, so the element $t = \pi_\omega(t_1,t_2,t_3, \dots) \in A_\omega$ has the desired properties. 
\end{proof}

\noindent
We start now the proof of Proposition \ref{prop:from.comparison.to.excision}:
\begin{definition}\label{def:one-step-elementary}
Let 
$A \subseteq B$ be \Cs s. A completely positive map  
$\varphi\colon A\to B$  is said to be \emph{one-step-elementary} if there exist a 
pure state $\lambda$ on $B$ and 
$d_1, \dots,d_n; c_1,\dots,c_n \in B$ such that 
$$\varphi (a) = \sum _{j,k=1}^n \lambda (d_j^*\,a\,d_k)c_j^*c_k\,.$$
\end{definition}

\noindent
An inspection of  the proof of \cite[Proposition 2.2]{Matui.Sato}, using our  Lemma \ref{lem:from.weak.comparison.to.estimate} instead of \cite[Lemma 2.5]{Matui.Sato}, gives a proof of the following: 

\begin{lemma}[cf.\ Proposition 2.2 of \cite{Matui.Sato}]
\label{lm:one-step-excision}
If $A$ is a unital, simple, and separable \Cs{} with the local weak comparison property and with $QT(A)=T(A) \ne \emptyset$,
then every  one-step-elementary  completely positive map  
$\varphi\colon A\to A$
can be excised in small central sequences.
\end{lemma}


\begin{lemma}\label{lem:excision.appox.invariant}
If $A$ is a  separable \Cs, then the family of
all  completely positive maps  
$\varphi\colon A\to A$,
that can be excised in small central sequences,
is closed under point-norm limits.
\end{lemma}

\begin{proof} Let $\varphi_n \colon A \to A$ be a sequence of completely positive maps, each of which can be excised in small central sequences, and which converges pointwise to a (completely positive) map $\varphi \colon A \to A$. We show that $\varphi$ can be excised in small central sequences. 

Let $e,f\in F(A)$ be given such that $e \in J_A$ and $\sup_k \|1-f^k\|_2 < 1$. For each $n \ge 1$ there exist $s_n \in A_\omega$ such that $fs_n = s_n$ and $s_n^*as_n  = \varphi_n(a)e$ for all $a \in A$. 

The Banach-Steinhaus  theorem
shows that the sequence $(\varphi_n)_{n \ge 1}$ is 
uniformly bounded. Let $(u_k)_{k \ge 1}$ be an increasing approximate unit for $A$ consisting of positive contractions (if $A$ is unital we can take $u_k = 1$ for all $k$). Upon replacing $\varphi_n$ with the map $a \mapsto \varphi_n(u_kau_k)$ for a suitably large $k=k(n)$, we can assume that $\|\varphi_n\| \to \|\varphi\|$.

Lift $e$ and $f$ to positive contractions $(e_1,e_2,\ldots)$ and $(f_1,f_2,\ldots)$, respectively, in $\ell^\infty(A)$. Let $(a_n)_{n \ge 1}$ be a dense sequence in the unit ball of $A$. We shall use
  Lemma \ref{lem:omega.selection} to finish the proof. Let each $X_n$ be the unit ball of $A$; and
consider the test functions
$$f^{(1)}_n(x) := \| (1-f_n) x\|, \qquad
f^{(k+1)}_n(x):= 
\| x^* a_k x - \varphi(a_k)e_n\|, \quad x \in X_n,  \; \; k \ge 1.
$$
Let $s'_m =(s_{m,1},s_{m,2},s_{m,3}, \dots) \in \ell^\infty(A)$ be a lift of $s_m$. Then $f_\omega^{(1)}(s'_m) = \|(1-f)s_m\| = 0$, and
$$f_\omega^{(k+1)}(s'_m) = \|s_m^*a_ks_m-\varphi(a_k)e\| = \|(\varphi_m(a_k)-\varphi(a_k))e\| \; \overset{m \to \infty}{\longrightarrow} \; 0$$
for all $k \ge 1$. The $\varepsilon$-test of Lemma \ref{lem:omega.selection} is thus fulfilled, and so there exist contractions $(t_n)_{n \ge 1}$ in $A$  such that 
$t:=\pi_\omega(t_1,t_2,\ldots)$
satisfies  $ft=t$ and $t^*at =\varphi(a)e$ for all $a\in A$.
\end{proof}

\noindent
The following useful observation (together with the lemma above) allow us to bypass  \cite[Chapter 3]{Matui.Sato}. In particular, our arguments do not depend on \cite[Lemma 3.3]{Matui.Sato}.

\begin{prop}\label{lem:approx.by.elementary}
Let $A \subseteq B$ be a separable \Cs s,  with $B$ simple and non-elementary, and let $\varphi\colon A\to B$ be a completely positive map. Then $\varphi$ is nuclear if and only if it 
is the point-norm limit of a sequence of 
completely positive one-step-elementary maps $\varphi_n\colon A\to B$.
\end{prop}

\noindent One can relax the assumptions on $B$ to the following: there exists a pure state $\lambda$ on $B$ such that the associated GNS representation $ \rho_\lambda \colon B \to \mathcal{L}(H)$ satisfies $ \rho_\lambda^{-1}(\cK(H)) = \{0\}$. 
\begin{proof} Each one-step-elementary map has finite dimensional range, and is therefore nuclear by \cite[Theorem 3.1]{ChoiEffros-78}. As the set of nuclear maps is closed under point-norm limits, we see that the "if" part of the proposition holds. 

Suppose now that $\varphi$ is nuclear. Then $\varphi$ can be approximated
in the point-norm topology  by maps of the form 
$V\circ U\colon A\to B$, where
$U\colon A\to M_n$ and $V\colon M_n\to B$
are completely positive maps, and $n \ge 1$. It thus suffices to show that $V \circ U$ can be approximated by one-step-elementary maps.
By Arveson extension theorem, the completely positive map $U$ extends
to a completely positive map $W\colon B\to M_n$.

Using a trick from the Stinespring theorem, we can decompose $V\colon M_n\to B$ 
as a superposition $V= T\circ E$, with
$$E\colon M_n\to M_n\otimes 1_n\subset M_{n^2}, \qquad E(x) = x \otimes 1_n,$$ 
and a 
completely positive map $T\colon M_{n^2}\to B$ of the form $T(y)= C^*y \, C$, for $y\in M_{n^2}$, where 
$C$ is a suitable column matrix in $M_{n^2,1}(B)$. Indeed, if $(e_{ij})_{i,j=1}^n$ are the matrix units for $M_n$, then $$P=(\mathrm{id}_n \otimes V)\Big(\sum_{i,j} e_{ij} \otimes e_{ij}\Big) \in M_{n} \otimes B$$ is positive, and hence has a positive square root $P^{1/2} = \sum_{i,j} e_{ij} \otimes c_{ij}$, with $c_{ij} \in B$. Then $V(e_{ij}) = \sum_{k=1}^n c_{ki}^* c_{kj}$ for all $i,j$. We obtain $C = (c_1, c_2, \dots, c_{n^2})^{T}$ from a suitable rearrangement of the matrix $(c_{ij})$. 

Now, $V\circ W=T\circ S$, where $S:=E\circ W \colon B \to M_{n^2}$ is completely positive; and $V \circ U = (T \circ S)|_A$. We have the following commutative diagram:
$$\xymatrix@C+2pc@R+2pc{B \ar@{-->}[drr]^-{\; \; S} \ar[dr]_W & A \ar[l]<-.3ex>_{\supseteq} \ar[r]^{V \circ \, U} \ar[d]_-{U}  & B \\
 & M_n \ar[r]_-E  \ar[ur]^V &  M_n \otimes M_n \ar[u]_T}$$

Choose a pure state $\lambda$ on $B$ such that $ \rho_\lambda^{-1}(\cK(H)) = \{0\}$. (If $B$ is simple and non-elementary, then any pure state will have this property.)

We show below that the map
$S\colon B\to M_{n^2}$ is the point-norm limit of maps $S' \colon B \to M_{n^2}$ of the form 
$$S' (b):= \big[ \lambda (d_j^*bd_k) \big]_{1\leq j,k\leq n^2}, \qquad b \in B,$$
where $d_1, d_2,\ldots, d_{n^2}$ are elements in $B$. This will finish the proof, since  $(T\circ S')|_A\colon A\to B$  is a one-step-elementary completely positive map, namely the one given by
$$(T\circ S')|_A (a)= \sum _{j,k}  \lambda(d_j^*ad_k)c_j^*c_k, \qquad a \in A,$$
and $V \circ U$ is the point-norm limit of maps of the form $(T \circ S')|_A$.

It is a consequence of Stinespring dilation theorem for the completely positive map $S \colon B \to M_{n^2}$, that there is a  representation $\rho \colon B \to \mathcal{L}(H_0)$ and vectors $\eta_1, \eta_2, \dots, \eta_{n^2} \in H_0$ such that $\big[S(b)\big]_{ij} = \langle \rho(b) \eta_j, \eta_i \rangle$. Upon replacing $\rho$ with $\rho \oplus \sigma$ for some non-degenerate representation $\sigma \colon B \to \mathcal{L}(H_1)$ with $\sigma^{-1}(\cK(H_1)) =\{0\}$, we can assume, moreover, that $\rho^{-1}(\cK(H_0)) =\{0\}$. It follows from a theorem of Voiculescu, \cite{Voiculescu.WvN.Thm}, see \cite{Arveson.Ext}, that $\rho$ and $ \rho_\lambda$ are approximately unitarily equivalent. We can therefore approximate $S$ in the point-norm topology by maps $S'\colon B\to M_{n^2}$ of the form $\big[S'(b)\big]_{ij}=\langle  \rho_\lambda(b)\xi_j,\xi_i \rangle$ for all  $b \in B$,
for some vectors $\xi_1,\xi_2, \dots, \xi_{n^2} \in  H$. Let $\xi_0 \in H$ be the canonical separating and cyclic vector representing the pure state $\lambda$. 
Kadison's transitivity theorem for irreducible representations provides us with
elements $d_1,d_2, \ldots,d_{n^2}\in B$ with
$ \rho_\lambda(d_j)\xi_0 = \xi_j$. Hence,
$$\big[S'(b)\big]_{ij}=\big\langle  \rho_\lambda(b)\xi_j,\xi_i \big\rangle = \big\langle  \rho_\lambda(bd_j)\xi_0,  \rho_\lambda(d_i)\xi_0 \big\rangle = \lambda (d_i^*bd_j),$$
for all $b \in B$, as desired.
\end{proof}

\begin{prop}\label{prop:from.comparison.to.excision}
Suppose that $A$ is  a unital, simple, and separable  \Cs{} with the local weak comparison property and with $QT(A) = T(A) \ne \emptyset$. Then every nuclear completely positive map $\varphi \colon A\to A$
can be excised in small central sequences.
\end{prop}

\begin{proof}
Combine Lemma \ref{lm:one-step-excision}, Lemma \ref{lem:excision.appox.invariant}, and Proposition \ref{lem:approx.by.elementary} (with $B=A$).
\end{proof}

\medskip
\noindent
Matui and Sato proved in  \cite{Matui.Sato} that if the identity map on a unital (simple, separable) \Cs{} can be excised in small central sequences, then the \Cs{} has property (SI). The same holds in our setting: If $\mathrm{id}_A \colon A \to A$ can be excised in small central sequences, then $A$ has property (SI). Indeed, in the ultrafilter notation of both properties, let $e,f \in F(A)$ be given with $e \in J_A$ and $\sup_k \|1-f^k\|_{1,\omega} < 1$. As $\mathrm{id}_A$ can be  excised in small central sequences there is $s \in A_\omega$ such that $fs=s$ and $s^*as = ae$ for all $a \in A$. It is then easy to see that $(as-sa)^*(as-sa)=0$ for all $a \in A$, so $s \in F(A)$.

Note that if $A$ is unital, stably finite, and exact, then $QT(A) = T(A) \ne \emptyset$. We therefore get the following:
\begin{cor}
\label{cor:from.excision.to.(SI).for.nuclear.A}
Suppose that $A$  is a unital, stably finite, simple, nuclear, and separable \Cs{} with the local weak comparison property. Then $A$ has property  (SI).
\end{cor}

\noindent We proceed to state a result that shows the importance of having property (SI). The implication "(iii) $\Rightarrow$ (iv)"  is implicitly contained in \cite{Matui.Sato}. The dimension drop \Cs{} $I(k,k+1)$ is the \Cs{} of all continuous functions $f \colon [0,1] \to M_k \otimes M_{k+1}$ such that $f(0) \in M_k \otimes \C$ and $f(1) \in \C \otimes M_{k+1}$. 

%
\begin{prop}\label{prop:(SI).and.unital.image.of.I(2,3).in.F(A).mod.JA}
Suppose that  $A$ is a separable, simple, unital and stably finite \Cs, and that $A$
has property (SI). Then the following properties are equivalent:
\begin{itemize}
\item[(i)] $A\cong A\otimes \mathcal{Z}$. \vspace{.1cm}
\item[(ii)]
There exists
a unital \sh{}  $\mathcal{R} \to  F(A)/(F(A)\cap J_A)$, where $\mathcal{R}$ denotes the hyperfinite II$_1$-factor.  \vspace{.1cm}
\item[(iii)]
There exists
a unital \sh{}  $M_k \to  F(A)/(F(A)\cap J_A)$
for some $k\ge 2$.  

\vspace {.1cm} \item[(iv)]
There exists
a unital \sh{}  $I(k,k+1) \to  F(A)$
for some $k\ge 2$. 
\end{itemize}
 \end{prop}
\begin{proof} (i) $\Rightarrow$ (ii). Assume that $A \cong A \otimes \cZ$. Then one can find 
an asymptotically central sequence of unital \sh s
$\psi_n\colon \mathcal{Z}\to A$ such that the image of the unital \sh{} 
$\psi_\omega\colon \mathcal{Z}_\omega \to A_\omega$
is contained in $F(A)$.
It follows from the definition of the trace-kernel ideal  that $\psi_\omega( J_\cZ) \subseteq  J_A$. The unital \sh{} $\psi_\omega$ therefore induces a unital \sh{}
$$F(\cZ)/(J_\cZ \cap F(\cZ)) \to F(A)/(J_A \cap F(A)).$$
 The Jiang-Su algebra has a unique tracial state $\tau$, so $J_\cZ = J_\tau$. It therefore follows from \cite[Lemma 2.1]{Sato.1104}  (or from our Theorem \ref{thm:M-omega}) that the hyperfinite II$_1$-factor $\mathcal{R}$ embeds unitally into $F(\cZ)/(J_\cZ \cap F(\cZ))$.

(ii) $\Rightarrow$ (iii) is trivial.

(iii) $\Rightarrow$ (iv). One can lift the unital \sh{} $M_k \to F(A)/(F(A) \cap J_A)$ to a (not necessarily unital)  completely positive order zero map $\varphi \colon M_k \to F(A)$. Put $a_j = \varphi(e_{jj})$, $1 \le j \le k$. Then $a_1,a_2, \dots, a_k$ are pairwise orthogonal, pairwise equivalent, positive contractions in $F(A)$. Moreover, $e:= 1-(a_1+a_2+ \cdots + a_k) \in J_A$, and $\|1-a_1^m\|_{1,\omega} = 1/k$ for all $m$.

By Lemma \ref{lem:from.weak.comparison.to.estimate} (i) there is $f \in F(A)$ such that $\sup_m\|1-f^m\|_{1,\omega} = \sup_k\|1-a_1^m\|_{1,\omega} = 1/k$ and $fa_1 = f$. Property (SI) implies that $e \precsim f$ in $F(A)$; and $f  \precsim (a_1-1/2)_+$ holds because $fa_1 = f$, so $e \precsim (a_1-1/2)_+$ in $F(A)$. The existence of a unital \sh{} $I(k,k+1) \to F(A)$ now follows from \cite[Proposition 5.1]{Ror.Wint}. 

(iv) $\Rightarrow$ (i). We have unital \sh s:
$$\cZ \to \bigotimes_{n=1}^\infty I(k,k+1) \to F(A).$$
The existence of the first \sh{} follows from Dadarlat and Toms, \cite{DadarlatToms.Z.stability}. As remarked in Section \ref{sec:intro}, the existence of the second \sh{} follows from \cite{Kir.AbelProc}, where it is shown that if there is a unital \sh{} $D \to F(A)$ for some separable unital \Cs{} $D$, then there is a unital \sh{} from the (maximal) infinite tensor power of $D$ into $F(A)$.  It is well-known that the existence of a unital \sh{} $\cZ \to F(A)$ implies that $A \cong A \otimes \cZ$ when $A$ is separable. 
\end{proof}

\noindent We conclude this section with an observation that may have independent interest:

\begin{prop}\label{prop:from.excision.to.(SI)}
Suppose that $A$ is a simple, separable, unital and stably finite \Cs{} with property (SI).  Then, for every positive contraction $c\in J_A$,
there exists $s\in J_A \cap F(A)$ with $s^*sc=c$ and $ss^*c=0$. 
In particular, $J_A$ and $J_A\cap F(A)$ do not have characters.
\end{prop}

\begin{proof} By the fact that 
$J_A$ is a $\sigma$-ideal  (applied to the separable sub-\Cs{} $C^* (A,c)$ of $A_\omega$) there is a positive contraction 
$e\in J_A \cap F(A)$ such that $ec=c=ce$. Let $f= 1-e$. Then $1-f^k\in J_A\cap F(A)$ for all integers $k \ge 1$, so $\sup_k \| 1-f^k \|_{2,\omega}= 0$. By property (SI),  there is
$s\in F(A)$ with $s^*s=e$ and $fs=s$. The former implies that $s \in J_A$ and $s^*sc=c$, and the latter implies that $s^*e=0$, so $s^* c = 0$.

It follows from the first part of the proposition that any \sh{} from $J_A$, or from $J_A \cap F(A)$, to the complex numbers must vanish on all positive contractions. Hence it must be zero. This proves that $J_A$ and $J_A \cap F(A)$ have no characters.
\end{proof}


\section{Affine functions on the trace simplex}
\label{sec:partial.T(A).closed}
\noindent
Let $A$ a unital \Cs{} with $T(A)\not=0$.
We denote the closure of $\partial T(A)$ in $T(A)$ by $bT(A)$. One can think of $bT(A)$ as some sort of \emph{Shilov boundary} (minimal closed norming set)  for the function system $\mathrm{Aff}_c(T(A))$, consisting of all continuous real valued affine functions on $T(A)$. 
The main topic of this section is to study the range of the canonical unital completely positive mapping
$A \to \Cf(bT(A))$, and of its ultrapower 
$A_\omega \to  \Cf(bT(A))_\omega$. (We shall denote these mappings by $\T$ and $\T_\omega$, respectively.) In particular we conclude, without assuming nuclearity of $A$, that for given  non-negative continuous functions 
$f_1,\ldots, f_n\colon bT(A)\to [0,1]$,  with 
disjoint supports, there exist pairwise orthogonal positive elements $a_1,\ldots, a_n$ in $A$ such that $\T(a_j)$ is close to $f_j$ for all $j$.

Let $S$ denote a compact convex set in a locally convex vector space $L$.
We denote by $\mathrm{Aff}_c(S)$ the space of all real-valued
continuous affine functions on $S$.  
We use in the following considerations
that the functions in $\mathrm{Aff}_c(S)$ of the form  $s \mapsto f(s)+\alpha$, with $\alpha\in \R$
and  $f$
a continuous linear functional on $L$, are uniformly dense 
in $\mathrm{Aff}_c(S)$. This can be seen by a simple separation argument. We denote the extreme points of $S$ by $\partial S$, and $bS$ denotes the closure of $\partial S$.

The space $\mathrm{Aff}(S)$ of \emph{bounded}\,\footnote{There exist
unbounded affine functions on the  simplex $S$ of states on 
$C(\{ 0\}\cup \{1/n:  n\in\N \})$ if one accepts the
axioms of choice for set theory.} affine functions
on $S$
consists of all pointwise
limits of bounded nets of functions in $\mathrm{Aff}_c(S)$, 
and $\mathrm{Aff}(S)$ coincides with $(\mathrm{Aff}_c(S ))^{**}$.

A classical theorem of Choquet, \cite[Corollary I.4.9]{Alfsen.boundary.int},
says that for each metrizable 
compact convex subset $S$ of a locally convex vector space $V$
and each $x\in S$ there exists a Borel probability measure
$\mu$ on  $S$, such that for all $f \in V^*$, 
\begin{equation} \label{eq:Choquet}
 f(x)= \int_{\partial S} f(s)\, \mathrm{d}\mu(s)
\end{equation}
and $\mu (\partial S)=1$.
Such a measure $\mu$ is called a 
Choquet \emph{boundary measure} for $x\in S$; and it is automatically regular. 
A compact convex subset $S$ of a locally convex vector space $V$
is a \emph{Choquet simplex} (also called 
\emph{simplex}
\footnote{\,There are definitions of simplexes $S$
by the Riesz decomposition
property for $S$, cf.\ \cite[prop.II.3.3]{Alfsen.boundary.int}.})
if the boundary measure $\mu$ in \eqref{eq:Choquet} is unique.

Every Borel probability measure $\mu$ on  $bS$, with
the property that $\mu (K)=0$ for every compact subset of 
$bS\setminus \partial S$, defines a state on $\Cf(bS)$,
and thus an element of $S$. It follows that 
two Borel probability measures $\mu$ and $\nu$ with
$\mu(\partial S)=1$ and $\nu(\partial S)=1$
define the same state on the algebra $\Cf(bS)$ if they
coincide on the real subspace $\mathrm{Aff}_c(S)\subseteq \Cf(bS)_{\mathrm{sa}}$.

We describe the map from $S$ to the family of boundary integrals
on $S$ with the help of the commutative  $W^*$-algebra $C(S)^{**}$, when $S$ is  a Choquet simplex.

Let $Q\in \Cf(bS)^{**}$ be the projection (in the countable up-down class)
that  corresponds to the G$_\delta$  subset $\partial S$ of 
$bS\subseteq S$.  Then $Q\leq P$, where $P\in \Cf(S)^{**}$
is the  (closed) projection 
corresponding to the closed set $bS$.
Then the (unique)  representation of the states on
$S$ as boundary integrals  extends
to an order-preserving isometric map
from $(\mathrm{Aff}_c(S))^*$  onto the predual
$Q\Cf(bS)^*= (Q\Cf(bS)^{**})_*$ of 
the commutative $W^*$-algebra $Q\Cf(bS)^{**}$.

It follows that the second conjugate order-unit
space $\mathrm{Aff}_c(S)^{**}$ is unitally
and isometric order isomorphic to the
self-adjoint part of the $W^*$-algebra $Q\Cf(bS)^{**}$. In other words: 
If $S$ is a Choquet simplex, then $\mathrm{Aff}_c(S)^{**}$ is 
a commutative $W^*$-algebra and the boundary
integral construction comes from a (unique)
extension of the natural embedding 
$\mathrm{Aff}_c(S)\subset \mathrm{Aff}_c(S)^{**}$
to a $^*$-homomorphism from $\Cf(bS)$ into  
$\mathrm{Aff}_c(S)^{**}$ (that coincides
with the natural $^*$-homomorphism from
$\Cf(bS)$ into $Q \Cf(bS)^{**}=Q\Cf(S)^{**}$).

Using  an obvious separation argument, one obtains the following
characterization
of Choquet--Bauer simplexes, \cite[Corollary II.4.2]{Alfsen.boundary.int}:
If $S$ is a metrizable  Choquet simplex with $\partial S$
closed in $S$, then
$S$ is the convex set of probability measures on
$\partial S$, and $\mathrm{Aff}_c(S)=\Cf(\partial S)_{\mathrm{sa}}$.
The order-unit space $\mathrm{Aff}_c(S )$
of real-valued continuous affine functions on a Choquet simplex $S $ is naturally
isomorphic to $\Cf(\partial S )_{\mathrm{sa}}=\Cf(\partial S , \R)$
if and only if $\partial S $ is closed in $S $.

\begin{lemma}\label{lem:NonExtreme}
Let $S$ denote a metrizable compact convex set,
and let $\tau_0 \in S\setminus \partial S$. Then there exists a compact subset $K\subseteq \partial S\setminus \{ \tau_0\}$
such that $f(\tau_0) < 2/3$ whenever $f\in \mathrm{Aff}_c(S)$ satisfies
$$0\leq f \leq 1, \qquad \sup_{\tau\in K} f(\tau) \le 1/3.$$
\end{lemma}

\begin{proof}
Let $\mu$ be a Choquet boundary measure for $\tau_0$. Then $\mu$ is a Radon probability  measure on $S$ satisfying $\mu(\partial S)=1$,
and
$$
f(\tau_0)=\int_{\partial S}  f(\tau)\, \mathrm{d} \mu(\tau)
$$
for all $f\in \mathrm{Aff}_c(S)$.  
By a theorem of Ulam, the finite measure $\mu$ is automatically inner regular, so there exists a compact set 
$K \subseteq \partial S$
such that $\mu (K) >  1/2$.
If $f\in \mathrm{Aff}_c(S)$ satisfies
$0\leq f \leq 1$ and
$f(\tau) \leq 1/3$ for all $\tau\in K$,
then 
$$ f(\tau_0)\;=\;\int_{\partial S}  f(\tau)\, \mathrm{d}\mu(\tau) \; \le \; \frac13 \cdot \mu(K) +(1- \mu(K))\; < \; 2/3. $$

\vspace{-1.0cm}
\end{proof}

\noindent We now return to the case where $S = T(A)$, the trace simplex of a unital \Cs{} $A$.
The order-unit space
of continuous \emph{real valued} affine functions on $T(A)$
will be denoted by $\mathrm{Aff}_c(T(A))$. The complexification
$\mathrm{Aff}_c(T(A))+i\mathrm{Aff}_c(T(A))$ of $\mathrm{Aff}_c(T(A))$, denoted
$\C\text{-}\mathrm{Aff}_c(T(A))$, can be viewed as the space
of complex valued affine
continuous functions on $T(A)$.
Define the norm $\| f\|$ on  $f\in \C\text{-}\mathrm{Aff}_c(T(A))$
by 
\begin{equation} \label{eq:norm-Aff}
\| f \| := \sup _{\tau\in T(A)}  | f(\tau) | = \sup _{\tau \in \partial T(A)} | f(\tau) |.
\end{equation}
It follows that $\C\text{-}\mathrm{Aff}_c(T(A))$ is a closed unital subspace
of $\Cf(T(A))$. Notice also that we have the following
natural unital isometric inclusions,
$$ 
\C\text{-}\mathrm{Aff}_c(T(A))\subseteq
\Cf(bT(A))\subseteq \Cb (\partial T(A)),
$$
where the latter is a $^*$-homomorphism.

Consider the natural unital completely positive map
 $$\T\colon A\to 
\C\text{-}\mathrm{Aff}_c (T(A))$$
defined by $\T(a)(\tau)=\tau(a)$ for $\tau\in T(A)$ and $a\in A$. We shall sometimes view $\T$ as a map from $A$ to $\Cf(bT(A))$. Note that $\T$ is \emph{central}, i.e., $\T(ab) = \T(ba)$ for all $a,b \in A$. Moreover, $\T(A_{\mathrm{sa}})=\mathrm{Aff}_c (T(A))$.

 \begin{lemma}\label{lem:Affc(T(A)).quotient.of.Asa}
 Let $A$ be a unital \Cs{} with $T(A)\not=\emptyset$.
 Denote the center of $A^{**}$ by $\cC$. Let $p$ denote the largest
 finite central projection of $A^{**}$, and let $E\colon A^{**}p \to \cC p$
be the (normal) center-valued  trace on the finite summand $A^{**}p$ of $A^{**}$.
 
The map $\T\colon A\to \Cf( bT(A))$ defined above  has the following properties:
 \begin{itemize}
 \item[(i)] 
Let $\Lambda$ denote the restriction of  $(\T)^{**}\colon A^{**}\to \Cf (bT(A))^{**}$ to $\cC p$. Then $\Lambda$  is a  unital  isometric isomorphism of (complexified) 
order-unit Banach spaces from $\cC p$ onto
$ \C\text{-}\mathrm{Aff}_c(T(A))$,
and $(\T)^{**}(a)= \Lambda(ap)$ for $a\in A^{**}$. \vspace{.1cm}
 \item[(ii)]
 $\T$ maps the open unit ball of $A$ onto 
 the open unit ball of $\C\text{-}\mathrm{Aff}_c(T(A))$.  \vspace{.1cm}

 \item[(iii)] $\T$ maps the open unit ball of $A_{\mathrm{sa}}$ onto the
 open unit ball of $\mathrm{Aff}_c(T(A))$.  \vspace{.1cm}
 \item[(iv)] $\T$  maps $A_+^1:=\{ a\in A : 0\leq a \leq 1\}$ onto a dense subset of
 the set of $f\in \mathrm{Aff}_c(T(A))$ with
 $f(T(A))\subseteq  [0,1]$.
 \end{itemize}
 \end{lemma}
 
\begin{proof}
Each central linear functional $\rho$ on $A$
 has a polar decompositions $\rho(a)=|\rho |(va)$, $a \in A$, where $|\rho|$ is a 
 positive central functional,  and where
 $v$ a partial isometry in the center of $A^{**}$.
 
The adjoint of the restriction of $\T$ to $A_{\mathrm{sa}}$ 
maps the unit ball of $(\mathrm{Aff}_c(T(A)))^*$, 
which is equal to  
$$\big\{ \alpha \tau_1 +\beta \tau_2 : \tau_1,\tau_2\in T(A),\; \alpha,\beta \in \R,\;
 |\alpha|+|\beta |\leq 1\big\},
$$
onto the set of all hermitian central linear functionals on $A$ of norm $\leq 1$.

The unit ball of $(\C\text{-}\mathrm{Aff}_c(T(A)))^*$
is equal to the 
norm closure   of   the absolute
convex hull  $\mathrm{conv}_\C (T(A))$  of $T(A)$.
This shows that
$$\T^*\colon (\C\text{-}\mathrm{Aff}_c(T(A)))^*\to A^*$$
maps the unit ball of $(\C\text{-}\mathrm{Aff}_c(T(A)))^*$
\emph{onto} the space of central linear
functionals in $A$ of norm $\leq 1$.

Let $\Delta_\R$ and $\Delta_\C$ denote the $\R$-linear, respectively, the $\C$-linear span of the commutator set $\{ i(ab-ba) : a,b\in A_{\mathrm{sa}} \}$. Then
 $$ \sup_ {\tau\in \partial T(A)} |\tau(a)| 
 = \mathrm{dist}(a, \Delta_\R), \qquad \sup_ {\tau\in \partial T(A)} |\tau(b)| 
 = \mathrm{dist}(b, \Delta_\C),$$
for $a \in A_{\mathrm{sa}}$ and $b \in A$. 
It follows that $\T \colon A\to \C\text{-}\mathrm{Aff}_c(T(A))$
defines a unital isometric isomorphism from
$A/[A,A]$ onto $\C\text{-}\mathrm{Aff}_c(T(A))$, where
$[A,A]$ denotes the norm closure of the
linear span of the self-adjoint commutators $i(ab-ba)$, where
$a,b\in A_{\mathrm{sa}}$.

 Since  $\T^*$ is an isometry
 which maps onto the
space of central functions on $A$, the map
$\T\colon A\to \C\text{-}\mathrm{Aff}_c(T(A))$
is a quotient map (i.e., 
$\T$ maps the open unit ball onto the open unit ball).

In a similar way one sees that the restriction of $\T$ to $A_{\mathrm{sa}}$ does the same for $A_{\mathrm{sa}}$
and $\mathrm{Aff}_c(T(A))$.  A separation argument finally shows that
$\T$ maps the positive contractions in $A$ onto a norm-dense
subset of the continuous affine maps $f\colon T(A)\to [0,1]$.
\end{proof}

\noindent Let $$\T_\omega \colon A_\omega \to (\C\text{-}\mathrm{Aff}_c(T(A)))_\omega$$ denote the ultrapower of the map $\T \colon A \to \C\text{-}\mathrm{Aff}_c(T(A))$.

\begin{lemma}\label{lem:Phi.on.A.omega}
Let $A$ be a unital \cst-algebra with $T(A)\not=\emptyset$. Then  $\T_\omega$ maps the \emph{closed} unit ball of $A_\omega$ onto the \emph{closed} unit ball of 
$(\C\text{-}\mathrm{Aff}_c(T(A)))_\omega$.
\end{lemma}

\begin{proof}
This follows from Lemma \ref{lem:Affc(T(A)).quotient.of.Asa} (ii), 
because every element  $f$ of the closed unit ball of 
$(\C\text{-}\mathrm{Aff}_c(T(A)))_\omega$
can be represented  as 
$$
f=\pi_\omega(f_1,f_2,\ldots) = \pi_\omega(\T(a_1),\T(a_2),\ldots)=
\T_\omega(\pi_\omega(a_1,a_2,\ldots))
$$
with each $f_n$ in the \emph{open} unit ball of $\C\text{-}\mathrm{Aff}_c(T(A))$
and $a_n$ in the open unit ball of $A$.
\end{proof}


\begin{lemma}\label{lem:averages.of.inner.automorphisms}
Let $A$ be a  \Cs, let $a_1,\ldots, a_m, b_1,\ldots, b_n\in A$, and let 
$\varepsilon>0$. 
Then there exist $N\in \N$ and self-adjoint elements $h_1,\ldots ,h_N$ in $A$
with $\| h_\ell \| <\pi$ such that, for $1\leq j \leq m$ and $1\leq k \leq m$,
$$ 
\Big\|  \big[ a_j,  N^{-1} \sum_{\ell =1}^N  \exp (-ih_\ell)b_k \exp (ih_\ell)\big]\Big\| 
<\varepsilon.
$$
In particular, if $A$ is unital and $T(A) \ne \emptyset$
then the unital completely positive map
$\T_\omega \colon  A_\omega \to 
(\C\text{-}\mathrm{Aff}_c (T(A)))_\omega$
maps the closed unit ball of $F(A)$ onto
the closed unit ball of $(\C\text{-}\mathrm{Aff}_c (T(A))_\omega$.
\end{lemma}

\begin{proof} Using a standard separation argument, one can reduce the first statement to the case of the $W^*$-algebra $A^{**}$. Here it follows from \cite[chp.III, \S 5, lem.4, p.253]{Dixmier.vN.alg.1969}, because each unitary element in each $W^*$-algebra $M$ is an exponential $\exp (iH)$ with $H^*=H\in M$  and $\| H\|\leq \pi$. Use also that the map $H\mapsto \exp(iH)$ is $^*$-ultra strongly  continuous,
and $\{h\in A: h^*=h, \; \| h \| <\pi \}$ is $^*$-ultra strongly
dense in $\{H\in A^{**} : H^*=H, \;  \| H \| \leq \pi\}$.

The second statement follows from the first and the fact that $\T$ is central. 
Thus, for each sequence $(b_1,b_2,\ldots)$ of contractions
in $A$ there exists a central sequence 
 $(c_1,c_2,\ldots)$ of contractions in $A$ such that
 $\pi_\omega (c_1,c_2,\ldots)\in F(A)$ and
 $\T_\omega(\pi_\omega (c_1,c_2,\ldots))=\T_\omega(\pi_\omega (b_1,b_2,\ldots))$.
 Lemma \ref{lem:Phi.on.A.omega} then
  shows that $\T_\omega$
 maps the closed unit ball of $F(A)$ onto the closed unit ball 
 of 
 $(\C\text{-}\mathrm{Aff}_c (T(A))_\omega$.
\end{proof}

\begin{lemma}\label{lem:Mult.Psi}
Suppose that $B$ and $C$ are unital \Cs s, 
$C$ is commutative and that $\Psi\colon B\to C$ is a 
faithful and central  unital completely positive map. Then the multiplicative domain, $\mathrm{Mult}(\Psi)$,
is a sub-\Cs{} of the center of $B$.

If $\Psi$ maps the closed unit ball of $B$
onto the closed unit ball of $C$,
then $\Psi$ defines an isomorphism from 
 $\mathrm{Mult}(\Psi)$ onto  $C$.
\end{lemma}
\begin{proof}
The multiplicative domain $\mathrm{Mult}(\Psi)\subseteq B$
of the unital completely positive map $\Psi$ is defined by the property
$\Psi(bv)=\Psi(b)\Psi(v)$ and $\Psi(vb)=\Psi(v)\Psi(b)$
for all $b\in B$
and $v\in \mathrm{Mult}(\Psi)$.
It clearly is a sub-\Cs{} of
$B$ which contains $1$. For all $a,b\in B$ we have
$$
\| \Psi (a^*b)-\Psi(a)^*\Psi(b)\|^2\leq 
\| \Psi (a^*a)-\Psi(a)^*\Psi(a)\|\cdot \| \Psi (b^*b)-\Psi(b)^*\Psi(b)\|
$$
(because $\Psi$ is unital and $2$-positive).
This shows that  $v\in \mathrm{Mult}(\Psi)$
if and only if $\Psi(v^*v)=\Psi(v)^*\Psi(v)$ and 
$\Psi(vv^*)=\Psi(v)\Psi(v)^*$.
The latter identity follows from the former if $\Psi$ is central and 
$C$ is commutative. 

The restriction of $\Psi$ to ${\mathrm{Mult}(\Psi)}$ defines an injective $^*$-homomorphism ${\mathrm{Mult}(\Psi)} \to C$ because $\Psi$ is faithful.

We show that $\mathrm{Mult}(\Psi)$ is contained in the center 
of $B$. Let $b^*=b\in B$ and let $u$ be a unitary element in $\mathrm{Mult}(\Psi)$. Then
$\Psi(u)^*\Psi(bub)=\Psi(u)^*\Psi(b^2u)=\Psi(b^2)$ because
$\Psi$ is central and
$C$ is commutative. It follows  that
$\Psi((bu-ub)^*(bu-ub))=0$. Thus, $bu=ub$. This shows that $\textrm{Mult}(\Psi)$ is 
contained in the
center of $B$.

Suppose that $\Psi$ maps the closed unit-ball
onto the closed unit ball of $C$.
If $u\in C$ is unitary, then
there exists a contraction $v\in B$ with $\Psi(v)=u$.
It follows that $1=\Psi(v)^*\Psi(v)\leq \Psi(v^*v)\leq 1$,
and $1=\Psi(v)\Psi(v)^*\leq \Psi(vv^*)\leq 1$.
Thus, $v$ belongs to $\textrm{Mult}(\Psi)$.
\end{proof}

\noindent In the proposition below, $\T_\omega \colon A_\omega \to \Cf(bT(A))_\omega$ is the "trace map" defined above Lemma \ref{lem:Phi.on.A.omega}. The multiplicative domain of $\T_\omega$ is denoted by $\mathrm{Mult}(\T_\omega)$. 


\begin{prop}\label{prop:char.partial.T(A).closed}
Let $A$ a unital and separable \Cs{} with 
$T(A)\ne \emptyset$.
Identify 
$\mathrm{Aff}_c(T(A))$ with a unital subspace
of $\Cf(bT(A), \R)=\Cf(bT(A))_{\mathrm{sa}}$, and identify $\C\text{-}\mathrm{Aff}_c(T(A))$ with a unital subspace
of $\Cf(bT(A))$.

The following properties of $A$ are equivalent:
\begin{itemize}
\item[(i)] For each $\tau_0\in bT(A)$ and 
for each non-empty compact set 
$K\subseteq bT(A)\setminus \{ \tau_0 \} $ 
there exists a positive contraction $a\in A$ with $$\tau_0(a)>2/3, \qquad \sup_{\tau\in K} \tau(a) \le 1/3.$$
\item[(ii)]  $\partial T(A)$  is closed in $T(A)$. \vspace{.1cm}
\item[(iii)] 
$\mathrm{Aff}_c(T(A))=\Cf(bT(A))_{\mathrm{sa}}$.  \vspace{.1cm}
\item[(iv)]
$\C\text{-}\mathrm{Aff}_c(T(A))=\Cf(\partial T(A))$.  \vspace{.1cm}
\item[(v)] $\T_\omega$ maps the closed unit ball of $A_\omega$ onto the closed
unit ball of $\Cf(bT(A))_\omega$.  \vspace{.1cm}
\item[(vi)] $J_A \subseteq \mathrm{Mult}(\T_\omega) \subseteq F(A)+J_A$, and $\T_\omega \colon \mathrm{Mult}(\T_\omega) \to  \Cf(bT(A))_\omega$ is a surjective $^*$-homomorphism with kernel equal to $J_A$.
\end{itemize}
\end{prop}

\begin{proof} The equivalences of  (ii), (iii), and (iv) have been discussed (for general Choquet simplexes $S$) above Lemma \ref{lem:NonExtreme}.  Lemma \ref{lem:NonExtreme} shows that (i) implies (ii), and Lemma \ref{lem:Phi.on.A.omega} shows that (iv) implies (v). 

(v) $\Rightarrow$ (vi). The closed two-sided ideal $J_A$ of $A_\omega$ consists of those  elements $a \in A_\omega$ for which $\T_\omega(a^*a)=0$. Therefore $\T_\omega$ factors through a faithful unital completely positive map $\Psi \colon A_\omega/J_A \to C(bT(A))_\omega$. If (v) holds, then Lemma \ref{lem:Mult.Psi} shows that $\mathrm{Mult}(\Psi)$ is contained in the center of $A_\omega/J_A$, and that $\Psi$ maps $\mathrm{Mult}(\Psi)$ onto  $\Cf(bT(A))_\omega$. The center of $A_\omega/J_A$ is contained in $F(A)/(F(A) \cap J_A) = (F(A)+J_A)/J_A$. 
The multiplicative domain of $\T_\omega$ is the preimage of $\mathrm{Mult}(\Psi)$ under the quotient mapping $A_\omega \to A_\omega/J_A$. Hence $\mathrm{Mult}(\T_\omega)$ is contained in $F(A) + J_A$, 
and $\T_\omega(\mathrm{Mult}(\T_\omega)) = \Psi(\mathrm{Mult}(\Psi)) = \Cf(bT(A))_\omega$.

(vi) $\Rightarrow$ (i).
Let $\tau_0\in bT(A)$ and $\emptyset \ne K\subseteq bT(A)$ be
a compact subset that does not contain $\tau_0$.
Then there is a continuous function $f\colon bT(A)\to [0,1]$
with $f(\tau_0)=1$ and $f(K)=\{ 0\}$ (by Tietze extension theorem).
By (vi) there is
a positive contraction $b:= \pi_\omega(b_1,b_2,\ldots)\in  \mathrm{Mult}(\T_\omega)$
with $\T_\omega(b)=f$. We can choose the lift $(b_n)_{n\ge 1}$ of $b$ to consist of positive contractions.
Now, $\lim_\omega \| \T(b_n) -f \|=0$, so
there exists $n\in \N$ such that
$$\sup_{\tau\in bT(A)} | f(\tau) -\tau(b_n)| <1/3.$$
The element $a = b_n$ is then a positive contraction satisfying $\tau_0(a)>2/3$ and $\sup_{\tau\in K} \tau(a) \le 1/3$.
\end{proof}

\begin{cor}\label{cor:char.partial.T(A).closed}
Let $A$ a unital and separable \Cs{} with 
$T(A)\ne \emptyset$. The following properties of $A$ are equivalent:
\begin{itemize}
\item[(i)]  $\partial T(A)$ is closed
in $T(A)$. \vspace{.1cm}
\item[(ii)] $\T_\omega \colon F(A) \to \Cf(bT(A))_\omega$
maps the closed unit ball of $F(A)$ onto the closed
unit ball of $\Cf(bT(A))_\omega$.  \vspace{.1cm}
\item[(iii)] For every separable unital sub-\Cs{}
$\cC$ of $C(bT(A))$ there exists
a sequence of unital completely positive maps  $V_n\colon \cC\to A$
such that, for all $c\in \cC$, 
$$\lim_{n\to \infty} \sup_{\tau\in \partial T(A)} 
\tau\big(V_n(c^*c)-V_n(c)^*V_n(c)\big) =0,  \qquad
\lim_{n\to \infty} \T(V_n(c)) =c.$$  
\item[(iv)]  
For every separable unital sub-\Cs s
$\cC$ of $\Cf(bT(A))$ there exists
a sequence of unital completely positive maps  $V_n\colon \cC\to A$
such that, for all $ a\in A$  and $c\in \cC$, 
$$\lim_{n\to \infty}\, \sup_{\tau\in \partial T(A)} 
\tau\big(V_n(c^*c)-V_n(c)^*V_n(c)\big) =0,\qquad
\lim_{n\to \infty} \T(V_n(c))=c,$$
and 
$$ \lim_{n\to \infty} \| [a, V_n(c) \| =0.$$
%
\end{itemize}
\end{cor}

\begin{proof}
The equivalence of (i) and (ii) follows from 
Proposition \ref{prop:char.partial.T(A).closed}.
Part (iv) clearly implies (iii); and  (iii) implies
Proposition \ref{prop:char.partial.T(A).closed} (v), which in turns implies part (i) of the present corollary.
It suffices to show that
Proposition \ref{prop:char.partial.T(A).closed} (vi) implies (iv). To do this we show that Proposition \ref{prop:char.partial.T(A).closed} (vi) 
implies that, for every compact subset $\Omega \subset A$,
every finite subset $X\subset \cC$, and every $\varepsilon >0$,
there exists a  unital completely positive map $V\colon \cC\to A$ with
$$\sup_{\tau\in \partial T(A)} 
\tau\big(V(c^*c)-V(c)^*V(c)\big) <\varepsilon, \qquad
\| \T(V(c))-c\| <\varepsilon, \qquad \| [a, V(c) \| <\varepsilon,$$  
for all $a\in \Omega$ and $c\in X$. 
 
Let $\Psi \colon \mathrm{Mult}(\T_\omega)/J_A\to \Cf(bT(A))_\omega$ be the factorization of $\T_\omega \colon  \mathrm{Mult}(\T_\omega) \to \Cf(bT(A))_\omega$  (as in the proof of Proposition \ref{prop:char.partial.T(A).closed}), and consider the $^*$-homomorphism $\Psi^{-1}|_\cC \colon \cC \to F(A)/(J_A\cap F(A))$. Take a unital completely positive lift $W\colon \cC\to F(A)$ of $\Psi^{-1}|_\cC$ (using that $\cC$ is nuclear).  This map can be further lifted to a unital completely positive map 
$$\widetilde{W} \colon \cC\to \pi_\omega^{-1}(F(A))\subset \ell^\infty(A).$$
Now $\widetilde{W}$ is given by a sequence of
unital completely positive maps $V_n\colon \cC\to \ell_\infty (A)$. For some integer $n \ge 1$, the map $V:= V_{n}$ has the desired properties.
\end{proof}


\begin{cor}\label{cor:partial.T(A).closed}
Let $A$ a separable, simple, and unital  \Cs{} with 
$T(A)\ne \emptyset$. Then $\partial T(A)$ is closed in $T(A)$
if and only if the following holds:

For any partition of the unit $(f^{(k)}_{i})_{1 \le k \le m, \; 1 \le i \le \mu(k)}$ of $\Cf(bT(A))$, 
with  $f^{(k)}_{i} f^{(k)}_{j}=0$
for all $k$ and for all $i\not=j$,  for every $\varepsilon >0$,
and for every compact subset $\Omega$ of $A$, there exist positive 
contractions $(a^{(k)}_i)_{1 \le k \le m, \; 1 \le i \le \mu(k)}$ in $A$ that satisfy the following conditions:
\begin{itemize}
\item[(i)] $a^{(k)}_i\, a^{(k)}_j=0$ for all $k$ and all $i,j$ with $i \ne j$. \vspace{.1cm}
\item[(ii)]  $\| [b,  a^{(k)}_j ] \| < \varepsilon$ for 
all $k,j$ and for all $b\in \Omega$. \vspace{.1cm}
\item[(iii)]  $| f^{(k)}_{j}(\tau)^{1/2} - \tau(a^{(k)}_{j})|\,  +\,
 | f^{(k)}_{j}(\tau) - \tau((a^{(k)}_{j})^2)| < \varepsilon$  
 for all $k,j$ and all $\tau \in bT(A)$.\vspace{.1cm}
\item[(iv)] $\sum_{k}  (a^{(k)})^2 \leq  1+\varepsilon$
where $a^{(k)}= a^{(k)}_1+\cdots +  a^{(k)}_{\mu(k)}$. \vspace{.1cm}
\item[(v)] $\| [ a^{(k)}_i, a^{(\ell)}_j ] \| <\varepsilon$ for all $k,\ell,i,j$.
\end{itemize}
\end{cor}

\begin{proof}  Property (iii) (applied to the partition of the unit $\{f,1-f\}$, with $m=1$ and $\mu(1)=2$, and where $f$ a positive contraction in $\Cf(bT(a))$ such that $f(\tau_0)=0$ and $f(K) = \{0\}$)
implies property (i) of Proposition \ref{prop:char.partial.T(A).closed}, and hence it implies $bT(A)=\partial T(A)$. This proves the "only if" part of the corollary. We proceed to prove the "if" part.

Suppose that  $\partial T(A)=bT(A)$ (i.e., that $\partial T(A)$ is closed). 
We shall use Proposition \ref{prop:char.partial.T(A).closed} (vi) (and its proof) to prove the existence of the positive contractions $a^{(k)}_j$. Let $\cC \subseteq \Cf(bT(A))$ be the separable unital sub-\Cs{} generated by the functions $(f_i^{(k)})$. As in the proof of  Proposition \ref{prop:char.partial.T(A).closed} (vi), let $\Psi \colon A_\omega/J_A \to \Cf(bT(A))_\omega$ be the (faithful) factorization of the map $\T_\omega \colon A_\omega \to \Cf(bT(A))_\omega$. We have a $^*$-isomorphism $\Psi \colon \mathrm{Mult}(\Psi) \to \Cf(bT(A))_\omega$, and $ \mathrm{Mult}(\Psi)  \subseteq F(A)/(F(A) \cap J_A)$. The inverse of this map, restricted to $\cC$, gives us an injective $^*$-homomorphism $\Phi \colon \cC \to F(A)/(F(A) \cap J_A)$. 
Since $J_A$ is a $\sigma$-ideal, we can apply \cite[Proposition 1.6]{Kir.AbelProc} to obtain an injective
$^*$-homomorphism $\varphi \colon \Cf_0((0,1], \cC)\to  F(A)$
with $\pi_{J_A} (\varphi (\iota\otimes c))= \Phi(c)$
for all  $c\in \cC$, where $\iota \in \Cf_0((0,1])$ as usual is given by $\iota(t) = t$.

Set $g^{(k)}_i =\varphi( \iota \otimes f^{(k)}_{i}) \in F(A)$. For each fixed $k$, we can lift the mutually orthogonal positive contractions $g^{(k)}_i$ to  mutually orthogonal positive contractions $b^{(k)}_i \in \ell_\infty(A)$, $1 \le i \le \mu(k)$. Let $P_n \colon \ell^\infty(A) \to A$ be the projection onto the $n$th copy of $A$. For a given $\varepsilon > 0$, the set $X \subseteq \N$ consisting of all $n \in \N$ for which 
the system, $a^{(k)}_i:= P_n(b^{(k)}_i)$, satisfies conditions (i)--(v) above,
is contained in the ultrafilter $\omega$. In particular, 
$X \ne \emptyset$, which completes the proof.
\end{proof}

\begin{rem}\label{rem:omega.characters}
Let $A$ be a unital \Cs. Put $X = bT(A)$. The Gelfand space of characters on $C(X)_\infty$  is the corona space
$$\gamma(X \times \N ):= \beta(X \times \N )\, \setminus\,  (X \times \N),  $$
where $\beta(X\times \N)$ denotes the Stone-Cech compactification of $X \times \N$. 

The Gelfand space of $\Cf(X)_\omega$ 
is a subspace of $\gamma(X \times \N )$ obtained as follows. The map $\pi_2\colon (x,n)\in X\times \N\mapsto n\in \N$ is 
continuous, open  and surjective.
Thus, it extends to a surjective continuous map $\beta(\pi_2)$ from
$\beta(X \times \N )$ onto $\beta (\N)$.
One can see that $\beta(\pi_2)^{-1}(\N)=X \times \N $,
because $X$ is compact.

It follows that $\beta(\pi_2)$ defines a surjective map
$\gamma(\pi_2)$ from $\gamma(X\times \N)$ onto
$\gamma(\N)$. 
The points of $\gamma(\N)$ are in natural bijective correspondence
to the \emph{free} ultra-filters on $\N$.  

The Gelfand space of $\Cf(X)_\omega$
is the space $Y:= \gamma(\pi_2)^{-1}(\omega)\subseteq \gamma(X\times \N)$.

Proposition \ref{prop:char.partial.T(A).closed} (vi)
and Lemma \ref{lem:Mult.Psi}
show that
the \Cs{} $A_\omega/J_A$ contains a copy of
$\Cf (bT(A))_\omega$ as a sub-\Cs{} of its center if $\partial T(A)=bT(A)$.

Thus $F(A)/(J_A\cap F(A))$ is a $\Cf(Y)$-\Cs. It would be interesting to know more about its fibers.
\end{rem}

\medskip

\section{The proof of the main result}
\label{sec:partial.T(A).fin.dim}

\noindent
In this section we prove our main result, Theorem \ref{thm:(SI).implies.Z-absorbtion}.  As in the previous sections, $bT(A)$ denotes the weak$^*$ closure of $\partial T(A)$ in $T(A)\subset A^*$, whenever $A$ is a unital \Cs{} (with $T(A) \ne \emptyset$). 

\begin{definition}\label{def:covering.system}
Let  $A$ be a unital \Cs, let $\Omega \subset A$ be a compact subset, and let $\varepsilon>0$. An \emph{$(\varepsilon, \Omega)$-commuting  covering system}  $(\mathcal{U}; \mathcal{V}^{(1)},\ldots, \mathcal{V}^{(m)})$ of size $m$ consists of the following ingredients:
\begin{itemize}
\item[(i)]
an open covering $\mathcal{U}$ of $bT(A)$, and \vspace{.1cm}
\item[(ii)] finite collections
$$\mathcal{V}^{(\ell)}=  \{\, V^{(\ell)}_1, V^{(\ell)}_2, \ldots, V^{(\ell)}_{\mu(\ell)}\,\}
$$
of unital completely positive maps  $V^{(\ell)}_j \colon M_2\to A$, $1 \le \ell \le m$ and $1 \le j \le \mu(\ell)$. 
\end{itemize}
The following properties must be satisfied:
\begin{itemize}
\item[(iii)]  $\| [ a, V^{(\ell)}_j(b)] \| <\varepsilon $ for all $a\in \Omega$, all contractions $b\in M_2$, and all $\ell$ and $j$.  \vspace{.1cm}
\item[(iv)] $\| [ V^{(\ell)}_j(b), V^{(k)}_i(c)] \| <\varepsilon $ for $\ell\not=k$, for all $i,j$, and for all
contractions $b,c\in M_2$.  \vspace{.1cm}
\item[(v)]  For each open set $U\in\mathcal{U}$ and each $\ell\in \{ 1,\ldots, m\}$
there exists $j\in \{ 1,\ldots, \mu(\ell) \}$ such that
$$\tau\big( V^{(\ell)}_j(b^*b)-  V^{(\ell)}_j(b)^*V^{(\ell)}_j(b)\big) < \varepsilon $$
for all contractions $b\in M_2$ and all $\tau \in U$.
\end{itemize}
\end{definition}

\begin{rem}\label{rem:refinement}
If $(\mathcal{U}; \mathcal{V}^{(1)},\ldots, \mathcal{V}^{(m)})$
is an $(\varepsilon, \Omega)$-commuting
covering systems of size $m$, and if
$\mathcal{U}'$ is a refinement of the covering $\mathcal{U}$,
then
$(\mathcal{U}'; \mathcal{V}^{(1)},\ldots, \mathcal{V}^{(m)})$
is again an  $(\varepsilon, \Omega)$-commuting
covering systems of size $m$.
\end{rem}

\noindent
Let $N$ be a $W^*$-algebra with separable predual, and let  $\tau$ a faithful
normal tracial state on $N$. As before, we let $N^\omega$ denote the $W^*$-algebra 
$\ell_\infty(N)/c_{\omega,\tau}(N)$, where 
$c_{\omega,\tau}(N)$
consists of the bounded sequences $(a_1,a_2,\cdots )$ with
$\lim_\omega \|a_n\|_{2,\tau}=0$. Recall further that for each $\tau \in T(A)$ we have a semi-norm $\| \; \cdot \; \|_{2,\tau}$ on $A_\omega$, and that $J_\tau$ is the closed two-sided ideal of $A_\omega$ consisting of those $a \in A_\omega$ for which $\|a\|_{2,\tau} = 0$. 

\begin{lemma}\label{lem:F(A).onto.Nc}
Let $A$ be a unital separable \Cs. Let $\tau\in T(A)$, and let $N_\tau = \rho_\tau(A)''$, where $\rho_\tau$ is the GNS representation of $A$ associated with $\tau$. If $N_\tau '\cap (N_\tau)^\omega$ contains a unital
copy of $M_2$, then there exists a sequence of unital completely positive maps
$V_n\colon M_2\to A$ such that,
$$
\lim_\omega \sup_{b\in (M_2)_1} 
\tau \big(V_n(b^*b)-V_n(b)^*V_n(b)\big) =0, \qquad \lim_\omega \; \sup_{a\in \Omega} \| [a , V_n(b)]  \| =0
$$
for each compact subset $\Omega$ of $A$, where $(M_2)_1$ denotes the unit ball of $M_2$.
\end{lemma}

\begin{proof}  It follows from Theorem \ref{thm:M-omega} that there is a surjective \sh{} $\psi \colon F(A) \to N_\tau '\cap (N_\tau)^\omega$ whose kernel is $J_\tau$. We can therefore lift the given unital \sh{} $M_2 \to N_\tau '\cap (N_\tau)^\omega$ to a unital completely positive map $\psi \colon M_2 \to F(A)$. The map $\psi$ lifts further to a unital completely positive map $V = (V_1,V_2, \dots ) \colon M_2 \to \ell^\infty(A)$. It is straightforward to check that the sequence of unital completely positive maps $V_n \colon M_2 \to A$ has the desired properties.
\end{proof}

\begin{prop}\label{prop:covering.system}
Let $A$ be a unital separable \Cs{} with $T(A) \ne \emptyset$. Suppose that, for each $\tau\in bT(A)$,
the corresponding $W^*$-algebra $N=\rho_\tau(A)''$ is McDuff, i.e., $N \, \overline{\otimes} \, \cR \cong N$. Let $m\in \N$. For each compact subset  $\Omega$ of $A$ and for each $\varepsilon>0$ there exists
an  $(\varepsilon, \Omega)$-commuting  covering system 
$(\mathcal{U}\,;\,\, \mathcal{V}^{(1)},\ldots, \mathcal{V}^{(m)})$
of size $m$ (in the sense of Definition \ref{def:covering.system}).
\end{prop}
\begin{proof}
We proceed by induction over the size $m$.
In case where $m=1$ we need only find one collection $\mathcal{V}^{(1)} = (V_j^{(1)})$
satisfying property (iii) and (v)  of Definition \ref{def:covering.system} with an  open covering $\mathcal{U}$
of $bT(A)$ (to be found).
Let $\tau\in bT(A)$. By the assumption on 
$N_\tau'\cap N_\tau^\omega$, Lemma \ref{lem:F(A).onto.Nc}
gives a unital completely positive map $V\colon M_2\to A$, that satisfies $\| [a,V(b) ]\| <\varepsilon$ for all  $a\in \Omega$ and for all $b$ in the unit ball of $M_2$.
It is easy to see that the map 
$$(b,\tau)\mapsto \tau\big(V(b^*b)-V(b)^*V(b)\big),  \qquad (b,\tau)\in (M_2)_1\times bT(A),$$
is uniformly continuous.
Thus there exists, for every $\varepsilon>0$ and every $\tau\in bT(A)$,
an open neighborhood $U_\tau$  of $\tau$ and a unital completely positive map
$V_\tau \colon M_2\to A$ so that 
$$\|[V_\tau(b),a]\| <\varepsilon, \qquad \tau' \big(V_{\tau}(b^*b)-V_{\tau}(b)^*V_{\tau}(b)\big) <\varepsilon$$
for all contractions $b \in M_2$, for all $a \in \Omega$, and 
for all $\tau'\in U_\tau$. Since $bT(A)\subseteq T(A)$ is compact,
we can find a finite subset $S$ of $bT(A)$ such that
$\mathcal{U}:= \{ U_\tau : \tau\in S \}$ and 
$\mathcal{V}^{(1)}:= \{ V_\tau : \tau\in S\}$ satisfy
the conditions (i), (ii), (iii), and  (v) of Definition \ref{def:covering.system}
for the given $\varepsilon>0$.

Suppose now that, for some $m \ge 1$ and for each given $\varepsilon>0$ and compact $\Omega \subset A$, there is 
an $(\varepsilon, \Omega)$-commuting covering systems $(\mathcal{U}; 
\mathcal{V}^{(1)},\ldots, \mathcal{V}^{(m)})$  of size $m$, where $\mathcal{V}^{(\ell)} = (V_j^{(\ell)})_{1 \le j \le \mu(\ell)}$. Then 
$$\Omega' = \Omega \, \cup \bigcup_{1 \le \ell\leq m, \, 1\le j\leq \mu(\ell)}  V^{(\ell)}_j\big((M_2)_1\big) \subset A$$
is compact, and so we can find a finite  $(\varepsilon, \Omega' )$-commuting
covering systems, $(\mathcal{U}', \mathcal{V}')$, of size 1. Put $\mathcal{V}^{(m+1)}:= \mathcal{V}'$, and
refine $\mathcal{U}$ and $\mathcal{U}'$ by taking all intersections $U'\cap U$
of open sets $U'\in \mathcal{U}'$ and $U\in \mathcal{U}$.
Then $(\mathcal{U}''; 
\mathcal{V}^{(1)},\ldots, \mathcal{V}^{(m)}, \mathcal{V}^{(m+1)})$
is an $(\varepsilon, \Omega)$-commuting
covering systems of size $m+1$, cf.\ Remark \ref{rem:refinement}.
\end{proof}
\begin{lemma}\label{lem:modification}
Let $A$ be a unital and separable \Cs{} with $T(A) \ne \emptyset$. Suppose further that $\partial T(A)$ is closed in $T(A)$, that $\partial T(A)$ has topological
dimension $m < \infty$, and that  for each $\tau\in \partial T(A)$
the corresponding II$_1$-factor $N=\rho_\tau(A)''$ is McDuff.

Then, for each compact subset $\Omega$
and for each $\varepsilon >0$, there exist
completely positive contractions $W^{(1)},\ldots,W^{(m+1)}\colon M_2\to A$
with the following properties:
\begin{itemize}
\item[(i)]    $\| [a,W^{(\ell)} (b)] \| <\varepsilon $ for 
 $\ell=1,\ldots, m+1$, for all $a\in \Omega$,
and for all contractions $b\in M_2$. \vspace{.15cm}
\item[(ii)] $\| [W^{(\ell)}(b),W^{(\ell)}(1)^{1/2}] \| <\varepsilon$ for 
 $\ell=1,\ldots, m+1$ and all contractions $b\in M_2$.  \vspace{.15cm}
\item[(iii)]  $\| [W^{(k)}(b),W^{(\ell)}(c)] \| <\varepsilon $ 
for all contractions $b,c\in M_2$
and all $k\not=\ell \in \{ 1,\ldots, m+1\}$.  \vspace{.15cm}
\item[(iv)] 
$\big\| W^{(\ell)} (1)^{1/2} W^{(\ell)} (b^*b)W^{(\ell)} (1)^{1/2}  -
W^{(\ell)}  (b)^*W^{(\ell)}  (b) \big\|_1
< \varepsilon$
for all $\ell=1,\ldots, m+1$ and all contractions $b\in M_2$. \vspace{.1cm}
\item[(v)] 
$\sup_{\tau\in bT(A)}\,\, \bigl|\, \tau\bigl(\sum_\ell  W^{(\ell)} (1) -1\bigr) \,\bigr| \, 
<\varepsilon$.  
\item[(vi)] $\sup_{\tau\in bT(A)} \,\,  
\tau \bigl(W^{(\ell)} (1)\bigr) - \tau \bigl(W^{(\ell)} (1)^{1/2}\bigr)^2    <\varepsilon$ for all $\ell$.
\end{itemize}
\end{lemma}

\begin{proof}
Let $\varepsilon > 0$ and $\Omega \subset A$ be given. Assume without loss of generality that $\ep < 1/2$.  Let $(\mathcal{U}; \mathcal{V}^{(1)},\ldots, \mathcal{V}^{(m)}, \mathcal{V}^{(m+1)})$ be
an $(\varepsilon/2, \Omega)$-commuting
covering systems of size $m+1$, cf.\ Proposition \ref{prop:covering.system}.
Since the covering dimension and the decomposition dimension
for compact metric spaces are 
the same\footnote{\, See  \cite[lem.3.2]{BlanKir.BanachBundles}
for an elementary proof of the equivalence of covering dimension
and decomposition dimension in case of normal topological spaces.}, 
we find
a refinement $\mathcal{U}'$  of the covering $\mathcal{U}$, such that
$\mathcal{U}'$ is the union of 
$m+1$ finite subsets $\mathcal{U}_1,\ldots, \mathcal{U}_{m+1}$,
and each $\mathcal{U}_\ell = (U^{\ell}_j)_{1 \le j \le \mu(\ell)}$ has the property, that any two
open sets $U^{(\ell)}_i$ and $U^{(\ell)}_j$ in $\mathcal{U}_\ell$ are disjoint if $i \ne j$. The number $\mu(\ell)$ is the cardinality of the set $\mathcal{U}_\ell$.

For each open set $U^{(\ell)}_j$ in $\mathcal{U}_j$ there is a unital completely positive map 
$V^{(\ell)}_j \colon M_2 \to A$ in $\mathcal{V}^{(\ell)}$ such that (iii), (iv) and (v) of Definition \ref{def:covering.system} hold, with (v) taking the concrete form:
$$ 
\sup_{\tau\in U^{(\ell)}_j}  
 \tau \big(V^{(\ell)}_j(b^*b) - V^{(\ell)}_j(b)^*V^{(\ell)}_j(b)\big)
 <\varepsilon/2,
 $$
for all contractions $b \in M_2$. Let  $(f^{(\ell )}_j)$ be a partition of the unit for $bT(A)$ subordinate to the open covering $\mathcal{U}' =
\mathcal{U}_1 \cup \cdots \cup \mathcal{U}_{m+1}$  of $bT(A)$.

The functions $(f^{(\ell )}_j)_{1\leq j \leq \mu(\ell )}$
are pairwise orthogonal for each fixed $\ell$. It follows from Corollary \ref{cor:partial.T(A).closed} that there exist positive contractions $(a^{(\ell)}_j)$ in $A$ such that (i)--(v) of  Corollary \ref{cor:partial.T(A).closed} holds with "$\ep = \ep_0$" and "$\Omega = \Omega_0$", where
$$\ep_0 = \varepsilon/G, \qquad 
G=\max\Big\{\sum_{k=1}^{m+1} \mu(\ell), \, 8 \cdot \max_{1 \le \ell \le m+1}  \mu(\ell)^2\Big\},$$
and 
$$
\Omega_{0} = 
\Omega \cup 
\big\{  V^{(\ell)}_j(b) :
b\in (M_2)_1, \, 1\leq \ell \leq m+1,\, 1\leq j\leq   \mu(\ell) \big\}.$$ 

Define completely positive contractions $W^{(\ell)}\colon M_2\to A$ by
$$
W^{(\ell )} (b)= 
\sum_{1\leq j\leq \mu(\ell)} \,  a^{(\ell)}_j V^{(\ell)}_j (b) a^{(\ell)}_j, \qquad b \in M_2, \quad 1 \le \ell \le m+1.
$$
We shall show that $W^{(1)}, W^{(2)}, \dots, W^{(m+1)}$ satisfy conditions (i)--(vi). The verifications are somewhat lengthy, but not difficult. We use that the number $\ep_0 > 0$ has been chosen such that for all $k, \ell$, 
$$4\mu(\ell) \ep_0\le \ep, \; \; 8\mu(\ell)\mu(k)\ep_0 \le \ep, \; \;  \ep_0 \sum_{k=1}^{m+1} \mu(\ell) \le \ep, \; \;  (3+\ep_0)\ep_0 \mu(\ell) \le \ep, \; \;  3\ep_0 \le \ep.$$
We shall  frequently use that
\begin{equation} \label{eq:W(1)}
W^{(\ell )}(1)= \sum_{1 \le j \le \mu(\ell)} (a^{(\ell)}_j)^2, \qquad W^{(\ell )}(1)^{1/2}= \sum_{1 \le j \le \mu(\ell)} a^{(\ell)}_j.
\end{equation}
The latter follows from the former together with the fact that $a^{(\ell)}_i \perp a^{(\ell)}_j$ when $i \ne j$.

Ad (i):  Fix $\ell = 1,2, \dots, m+1$. We show  that  $\| [a,W^{(\ell)}  (b)] \| <\varepsilon $ for  all $a\in \Omega$ and for all contractions $b\in M_2$. Write $ [a,W^{(\ell)}  (b)]  = S_1+S_2+S_3$, where
$$ S_1 = \sum_{j=1}^{\mu(\ell)}   [a, a^{(\ell)}_j] \cdot V^{(\ell)}_j (b) a^{(\ell)}_j , \; \;
S_2 = \sum_{j=1}^{\mu(\ell)}   a^{(\ell)}_j [a,V^{(\ell)}_j (b)] a^{(\ell)}_j, \; \;
S_3 =\sum_{j=1}^{\mu(\ell)}   a^{(\ell)}_j V^{(\ell)}_j (b)\cdot [a, a^{(\ell)}_j]. $$
Since the summands in $S_2$ are pairwise orthogonal, and since
$V^{(\ell)}_j$ is part of an $(\varepsilon/2 , \Omega)$-commuting covering, we get
$$
\| S_2 \|=
\sup_j  \| a^{(\ell)}_j [a,V^{(\ell)}_j (b)] a^{(\ell)}_j\|  <  \varepsilon/2.
$$
Since $ a^{(\ell)}_j$ and $b$ are contractions and $\|[a,a^{(\ell)}_j ]\| \le \ep_0$, it follows that
$\|S_\iota\| \le \mu(\ell) \cdot \ep_0$ for $\iota =1,3$.
We therefore get 
$$\| [a ,W^{(\ell)}  (b)] \| \le \|S_1\| +\|S_2\| +\|S_3\| <  \varepsilon/2    +  2\mu(\ell) \cdot \ep_0 \leq \varepsilon.$$

Ad (ii):  Write
 $$ 
 [W^{(\ell)} (b),W^{(\ell)} (1)^{1/2}]= 
 \sum_{j=1}^{\mu(\ell)}  a^{(\ell)}_j [a^{(\ell)}_j,V^{(\ell)}_j (b)] a^{(\ell)}_j.
 $$
Since $V^{(\ell)}_j(b) \in \Omega_0$, and because the elements 
 $a^{(\ell)}_j$ are pairwise orthogonal contractions, we see that the right-hand side is at most $\ep_0$, which again is smaller than $\ep$.

Ad (iii): Write $[W^{(k)}(b), W^{(\ell)}(c)]  = S_1+S_2+S_3$, where
\begin{eqnarray*}
S_{1} & = & \sum_{i=1}^{\mu(k)} \sum _{j=1}^{\mu(\ell)} \Big(
a^{(k)}_i V^{(k)}_i(b) a^{(k)}_i  a^{(\ell)}_j V^{(\ell)}_j (c)a^{(\ell)}_j
- 
a^{(k)}_i a^{(\ell)}_jV^{(k)}_i(b)V^{(\ell)}_j (c) a^{(\ell)}_j a^{(k)}_i \Big), \\
S_2  & = & 
\sum_{i=1}^{\mu(k)} \sum _{j=1}^{\mu(\ell)}    a^{(k)}_i a^{(\ell)}_j \big[V^{(k)}_i(b),V^{(\ell)}_j (c)\big] a^{(\ell)}_j a^{(k)}_i,\\
S_{3} & = &\sum_{i=1}^{\mu(k)} \sum _{j=1}^{\mu(\ell)} 
\Big(a^{(k)}_i a^{(\ell)}_j V^{(\ell)}_j (c) V^{(k)}_i(b) a^{(\ell)}_j a^{(k)}_i 
-   
a^{(\ell)}_j V^{(\ell)}_j (c)a^{(\ell)}_j a^{(k)}_i V^{(k)}_i(b)a^{(k)}_i\Big).
\end{eqnarray*}
We estimate each of $\|S_j\|$, $j=1,2,3$, below. The pairwise orthogonality of the elements 
$(a^{(\ell)}_j)_{1 \le j \le \mu(\ell)}$ and of the elements $(a^{(k)}_i)_{1 \le i \le \mu(k)}$ implies that 
\begin{eqnarray*}
\| S_2 \| &= &
\| \sum_{i}   a^{(k)}_i \Big( \sum_j a^{(\ell)}_j [V^{(k)}_i(b),V^{(\ell)}_j (c)] a^{(\ell)}_j \Big) a^{(k)}_i \| 
\\ &\leq & \max_{i} \| \sum_j a^{(\ell)}_j [V^{(k)}_i(b),V^{(\ell)}_j (c)] a^{(\ell)}_j\| \\
&\leq &   \max_{i,j} 
\| [V^{(k)}_i(b),V^{(\ell)}_j (c)]\| \; < \; \varepsilon/2.
\end{eqnarray*}
By construction, the positive contractions $a^{(k)}_i$ and $a^{(\ell)}_j$ commute with each other, and with all elements of $\Omega_0$, up to $\ep_0$. It follows easily from this that $\|S_\iota\| \le 2\mu(\ell)\mu(k)\ep_0$, $\iota=1,3$. We conclude that $\|[W^{(k)}(b), W^{(\ell)}(c)]\| < \ep/2+4\mu(\ell)\mu(k)\ep_0 < \ep$.

Ad (iv): Since we only need to consider elements from one collection $\mathcal{V}^{(\ell)}$ we drop the upper index and write $\mu = \mu(\ell)$, $a_j = a_j^{(\ell)}$, $V_j = V_j^{(\ell)}$, and $U_j = U_j^{(\ell)}$. Since  $W^{(\ell)} (1)^{1/2}= \sum_j a_j $, cf.\ \eqref{eq:W(1)}, we have
$$
W^{(\ell)} (1)^{1/2} W^{(\ell)} (b^*b)W^{(\ell)} (1)^{1/2}  -
W^{(\ell)}  (b)^*W^{(\ell)}  (b) =
\sum_{j=1}^{\mu} a_j T_j a_j,
$$
where $T_j = a_jV_j(b^*b)a_j  - V_j(b)^*a_j^2 \, V_j(b)$. 
It thus suffices to show that $\|a_jT_ja_j\|_1 < \ep/\mu$ for all $j$. Now, in the notation of Definition \ref{def:seminorms}, 
$$\|a_jT_ja_j\|_1 = \max_{1 \le i \le \mu} \|a_jT_ja_j\|_{1,U_i}.$$
Let us first consider the case where $i \ne j$. As $T_j$ is a contraction we have $\|a_jT_ja_j\|_{1,U_i} \le \|a_j\|_{1,U_i}$. By Corollary \ref{cor:partial.T(A).closed} (iii) and the fact that $f_j^{(\ell)}(\tau) = 0$ when $\tau \in U_i$, we conclude that $\|a_j\|_{1,U_i} < \ep_0 \le \ep/\mu$. Consider next the case $i=j$.
Write $T_j = a_jC_ja_j-R_j$, where 
$$
R_j=V_j(b)^*a_j^2 \, V_j(b) - a_jV_j(b)^*V_j(b)a_j, \qquad  C_j=V_j(b^*b)-V_j(b)^*V_j(b)\ge 0.
$$
Then  $\| R_j \| < 2\ep_0 $, because all $a_j$ commute with all elements of $\Omega_0$ within $\ep_0$, and $\|C_j\|_{1,U_j}< \ep_0$, by Definition \ref{def:covering.system} (v). This shows that
\begin{eqnarray*}
\|a_jT_ja_j\|_{1,U_j} & \le & \|a_j^2 C_ja_j^2\|_{1,U_j} + \|a_jR_ja_j\|_{1,U_j} \\
& \le & \|C_j^{1/2}a_j^4C_j^{1/2}\|_{1,U_j} + \|a_jR_ja_j\| \\
& \le & \|C_j\|_{1,U_j} + \|R_j\| \; < \; 3\ep_0 \; \le \; \ep/\mu.
\end{eqnarray*}

Ad (v): Use that $( f^{(\ell)}_j)$ is a partition of the unit, that $|f^{(\ell)}_j(\tau) - \tau\big((a^{(\ell)}_j)^2\big)| < \ep_0$ for all $\tau \in bT(A)$, and \eqref{eq:W(1)} to see that
\begin{eqnarray*}
 \big| \tau\big(\sum_{\ell=1}^{m+1}  W^{(\ell)} (1) -1\big) \big|  &\leq &
 \sum_{\ell=1}^{m+1} \, \sum_{j=1}^{\mu(\ell)}  \big|\tau\big((a^{(\ell)}_j)^2)- f^{(\ell)}_j(\tau)\big|
\; \leq  \; \ep_0 \cdot \! \sum_{\ell =1}^{m+1} \mu(\ell) \; < \; \varepsilon.
\end{eqnarray*}

Ad (vi):  Retaining the convention  from the proof of (iv) of dropping the superscript, we have
$W^{(\ell)} (1)= \sum _j   a_j^2$ and $W^{(\ell)} (1)^{1/2}= \sum _j   a_j$, cf.\ Equation \eqref{eq:W(1)}. Moreover, since the functions $(f_j)$ are pairwise orthogonal, we have $(\sum_j f_j^{1/2})^2 = \sum_j f_j$. Using property (iii) of 
Corollary~\ref{cor:partial.T(A).closed} we get, for each $\tau \in bT(A)$, 
\begin{eqnarray*}
&& \tau(W^{\ell}(1)) - \tau\big(W^{\ell}(1)^{1/2}\big)^2 \\ 
&\le& \big|\sum_{j=1}^{ \mu} \tau(a_j^2) - \sum_{j=1}^{ \mu} f_j(\tau)\big| + \big| \big(\sum_{j=1}^{ \mu} f_j(\tau)^{1/2}\big)^2 - 
\big(\sum_{j=1}^{ \mu} \tau(a_j)\big)^2\big| \\
&\le & \mu \cdot \ep_0+(2+\ep_0)\big| \sum_{j=1}^{ \mu} f_j(\tau)^{1/2} - \sum_{j=1}^{ \mu}  \tau(a_j) \big| \\
&< &  \mu \cdot \ep_0 +  (2+\ep_0) \cdot  \mu \cdot \ep_0 \; \le \;  (3+\ep_0)\ep_0  \mu 
\; \le \; \ep.
\end{eqnarray*}
We have used that $\sum_j f_j(\tau)^{1/2}  \le 1$ and that $\sum_j \tau(a_j) \le 1+\ep_0$. The latter estimate follows from Corollary~\ref{cor:partial.T(A).closed} (iv), which implies that $\|\sum_j a_j\| \le  (1+\ep_0)^{1/2}$. \end{proof}

\noindent The following lemma is almost contained in \cite[Lemma 5.7]{HirWinZac:Rokhlin}. 
However, as the statement in \cite{HirWinZac:Rokhlin} formally is different from the fact that we shall need, we provide a proof of our lemma for the convenience of the reader. 

\begin{lemma}\label{lem:commuting.order.zero.c.p.c}
Let $B$ be a \Cs, let $p \ge 2$, and let $\varphi_1,\ldots, \varphi_n\colon M_p\to B$
completely positive order zero maps whose images commute, i.e., 
$\varphi_k (b)\varphi_\ell (c)=\varphi_\ell (c)\varphi_k (b)$ when 
$k\ne \ell$ and $b,c\in M_p$.
\begin{itemize}
\item[(i)] If $\| \varphi_1(1)+\varphi_2(1)+\cdots +\varphi_n(1)\| \leq 1$,
then there exists a completely positive order
zero map $\psi\colon M_p\to B$ 
with 
$$\psi(1) = \varphi_1(1)+\varphi_2(1)+\cdots +\varphi_n(1)$$ 
such that $\psi(M_p)$
is contained in the sub-\Cs{} of $B$ that
is generated by the images $\varphi_k(M_p)$,
$k=1,\ldots,n$.
\item[(ii)] Suppose that $B$ is unital and that
$\varphi_1(1)+\varphi_2(1)+\cdots +\varphi_n(1)=1$.
Then there is a unital $^*$-homomorphism $\psi \colon M_p \to B$.
\end{itemize}
\end{lemma}
\begin{proof} (i). It clearly suffices to consider the case where $n=2$. Upon replacing $B$ by the \Cs{} generated by the images of the completely positive maps $\varphi_j$ we can assure that the image of $\psi$ will be contained in that sub-\Cs. 

We use that $\varphi_k(a)=\lambda_k(\iota \otimes a)$ for a \sh{} 
$\lambda_k\colon  \Cf_0 ((0,1], M_p)\to B$, where $\iota \in \Cf_0((0,1])$ denotes the function $\iota(t)=t$.
Since $\varphi_k(1)$ commutes with $\varphi_k(M_p)$,
and since $\varphi_1(M_p)$ and $\varphi_2(M_p)$ commute element-wise,
we get that $\varphi_1(1)+\varphi_2(1)$ is a strictly positive contraction
in the center of the sub-\Cs{} $B_0 := C^*(\varphi_1(M_p)\cup \varphi_2(M_p))$ of $B$.
Define new order zero completely positive contractions 
$\psi_k\colon M_p\to \cM (B_0)$, $k=1,2$, by 
$$ \psi_k(a):= \lim_{n\to \infty}\varphi_k(a)( \varphi_1(1)+\varphi_2(1)+n^{-1} \! \cdot \! 1)^{-1}$$
Then $\psi_1$ and $\psi_2$ are two commuting order zero maps 
with $\psi_1(1)+\psi_2(1)=1$.
Consider the corresponding  \sh s
$\Lambda_k\colon \Cf_0 ((0,1], M_p)\to \cM(B_0)$, i.e., where $\psi_k(a) = \Lambda_k(\iota \otimes a)$. Then $\Lambda_1$ and $\Lambda_2$ commute and 
satisfy $\Lambda_1(\iota \ot  1)+\Lambda_2(\iota\ot 1)= \psi_1(1) + \psi_2(1)=1$.

We show below that there is a unital \sh{} $\Lambda \colon M_p \to \cM(B_0)$ arising as the composition of unital \sh s:
\begin{equation} \label{eq:I(p,p)}
M_p \to I(p,p) \to C^*(\mathrm{Im}(\Lambda_1),\mathrm{Im}(\Lambda_2)) \hookrightarrow \cM(B_0),
\end{equation}
where $I(p,p)$ is the dimension drop \Cs, consisting of all continuous functions $f \colon [0,1] \to M_p \otimes M_p$ such that $f(0) \in M_p \otimes \C$ and $f(1) \in \C \otimes M_p$. 
Once the existence of $\Lambda$ has been established it will follow that $\psi(x) = \Lambda(x)(\varphi_1(1)+\varphi_2(1))$, $x \in M_p$, defines a  completely positive order zero map $\psi \colon M_p \to B$ satisfying $\psi(1) = \varphi_1(1) + \varphi_2(1)$.

The existence of the unital \sh{} $M_p \to I(p,p)$ in \eqref{eq:I(p,p)} was shown in \cite[proof of Lemma 2.2]{JiangSu.alg}. 

Let $\mathrm{Cone}(M_p)$ denote the unitization of the cone $\Cf_0((0,1],M_p)$. 
The \sh s $\Lambda_1$ and $\Lambda_2$ extend to unital \sh s $\mathrm{Cone}(M_p) \to \cM(B_0)$ with commuting images, and thus they induce a unital \sh{} $\mathrm{Cone}(M_p) \otimes \mathrm{Cone}(M_p)  \to \cM(B_0)$, which maps $1\otimes 1 - (\iota\otimes 1_p)\otimes 1 - 1\otimes (\iota\otimes 1_p)$ to zero, because  $\Lambda_1(\iota \ot  1)+\Lambda_2(\iota\ot 1)=1$. It is well-known that $I(p,p)$ is naturally  isomorphic to the quotient of $\mathrm{Cone}(M_p) \otimes \mathrm{Cone}(M_p)$ by the ideal generated by $1\otimes 1 - (\iota\otimes 1_p)\otimes 1 - 1\otimes (\iota\otimes 1_p)$. This proves the existence of the second \sh{} in \eqref{eq:I(p,p)}.

(ii). Let $\lambda\colon \Cf_0 ((0,1], M_p)\to B$ be the \sh{} associated with the unital completely positive order zero map $\psi \colon M_p \to B$, so that $\psi(a)=\lambda (\iota \otimes a)$.
Then $\psi(a^*a)\psi(1)=\psi(a)^*\psi(a)$ for all $a\in M_p$.
Hence, $\psi$ is multiplicative if $\psi(1)=1$.
\end{proof}

\begin{prop}\label{prop:M2.in.F(A).mod.JA}
Suppose that $A$ is a unital, stably finite, and separable \Cs{} such that $\partial T(A)$ is closed and has finite topological dimension, and such that for each $\tau\in \partial T(A)$
the corresponding II$_1$-factor $\rho_\tau(A)''$ is McDuff.

Then there exists a unital \sh{}  $\varphi\colon M_2\to F(A)/(J_A\cap F(A))$.
\end{prop}

\begin{proof} Let $\Omega \subset A$ be a compact subset of $A$
whose linear span is dense in $A$.
Find completely positive contractions  
$W^{(1)}_n,\ldots, W^{(m+1)}_n\colon M_2\to A$
that satisfy conditions (i) -- (vi) of Lemma
\ref{lem:modification} for 
$\varepsilon:= \varepsilon_n:=2^{-(n+1)}$ and $\Omega$.
Then the completely positive contractions $\lambda_k \colon M_2 \to A_\omega$, $1 \le k \le m+1$, given by 
$$\lambda_k (b):= \pi_\omega\big(W^{(k)}_1(b), W^{(k)}_2(b),W^{(k)}_3(b),\ldots\big), \qquad b \in M_2,$$
have image in $F(A)$, 
and they satisfy  
\begin{itemize}
\item $\big[\lambda _k(1), \lambda _k(b)\big] = 0$ for all $k$ and for all $b\in M_2$,  \vspace{.1cm}
\item $\big[\lambda_k(b),\lambda_\ell(c)\big] = 0$ for $k\not=\ell$ and $b,c\in M_2$, \vspace{.1cm}
\item $\lambda_k(1)\lambda_k(b^*b) -  \lambda_k(b)^*\lambda_k(b) \in J_A$ for all $k$ and all $b \in M_2$, \vspace{.1cm}
\item $\displaystyle{\T_\omega\big(\sum_{k=1}^{m+1} \lambda_k(1) - 1\big) = 0}$,  \vspace{.1cm}
\item $\lambda_k(1) \in \mathrm{Mult}(\T_\omega)$ for all $k$.
\end{itemize}
(To see that the third bullet holds, recall that $J_A$ consists of all elements $x  \in A_\omega$ with $\|x\|_{1,\omega} =0$.) It follows from the fourth and the fifth bullet that 
$$\T_\omega\Big( \big| \sum_{k=1}^{m+1} \lambda_k(1) - 1\big|\Big) = 0,$$
or, equivalently, that $\sum_{k=1}^{m+1} \lambda_k(1) - 1 \in J_A$. Thus,  
$$\varphi_k:=\pi_{J_A}\circ \lambda_k \colon M_2 \to F(A)/(J_A\cap F(A)), \qquad 1 \le k \le m+1,$$ define  order zero completely positive contractions with $\sum_{k=1}^{m+1} \varphi_k(1) = 1$. (Use the first and the third bullet above to see that the $\varphi_k$'s preserve orthogonality, so they are of order zero.)
Now, apply Lemma 
\ref{lem:commuting.order.zero.c.p.c} to obtain the desired unital \sh{} from $M_2$ into 
$F(A)/(J_A\cap F(A))$.
\end{proof}
\noindent
Our main result below extends \cite[Theorem 1.1]{Matui.Sato}. It follows immediately from Proposition \ref{prop:M2.in.F(A).mod.JA} above and from Proposition \ref{prop:(SI).and.unital.image.of.I(2,3).in.F(A).mod.JA}. 
\begin{thm}\label{thm:(SI).implies.Z-absorbtion}
Let $A$ be a non-elementary, unital, stably finite, simple, and separable  \Cs.
If $A$ satisfies conditions (i)--(iv) below, then 
$A\cong A\otimes \mathcal{Z}$.
\begin{itemize}
\item[(i)] Each II$_1$-factor representation of $A$
generates a McDuff factor.\footnote{\,
Equivalently: 
$F(A)/ (F(A) \cap J_{\tau,\omega})$ is non-commutative for every factorial trace states
$\tau$ of $A$, \cf \cite{McDuff.1970}.} \vspace{.1cm}

\item[(ii)] $A$ satisfies property $(SI)$, cf.\ Definition \ref{def:property(SI)}. \vspace{.1cm}

\item[(iii)]  $\partial T(A)$ is closed in $T(A)$, i.e., the Choquet simplex $T(A)$ is a Bauer simplex. \vspace{.1cm}
\item[(iv)] The set $\partial T(A)$ of extremal tracial states
is a topological space of finite dimension. 
\end{itemize}
\end{thm}

\noindent Some of the conditions in the theorem above are sometimes automatically fulfilled. 
Condition (i) holds for all nuclear \Cs s (cf.\ Theorem \ref{thm:M-omega} and the theorem by Connes that hyperfinite II$_1$-factors are McDuff). Condition (i) also holds
for all $A$ with $A\cong A\otimes \mathcal{Z}$. Hence condition (i) is \emph{necessary}. 

Condition (ii) holds for all \emph{nuclear} \Cs s $A$ that have \emph{local weak comparison}
(which again is implied by $\alpha$-comparison for some $\alpha<\infty$,
and in particular from strict  comparison). See Corollary \ref{cor:from.excision.to.(SI).for.nuclear.A}. 

It is unknown, if  $A\cong A\otimes \mathcal{Z}$ implies 
property (SI), even for separable, simple, unital and \emph{exact} \Cs s
$A$.

Conditions (iii) and (iv) do not follow from having the isomorphism $A\cong A\otimes \mathcal{Z}$. Indeed, 
any (metrizable compact) Choquet simplex can appear as 
$T(A)$ of a suitable simple separable unital AF-algebra
$A$,   \cite{EffrosHandelmanShen}. 
Infinite-dimensional simple separable unital AF-algebras 
always tensorially  absorb the Jiang-Su algebra $\mathcal{Z}$,
because they are approximately divisible.

\smallskip
\begin{cor}\label{cor:nuclear.case}
The following conditions are equivalent for any non-elementary
stably finite, separable, nuclear, simple and unital \Cs{} $A$, for which $\partial T(A)$ is closed in $T(A)$
and $\partial T(A)$  is a finite-dimensional topological space.
\begin{itemize}
\item[(i)] $A \cong A \otimes \cZ$. \vspace{.1cm}
\item[(ii)] $A$ has the local weak comparison property. \vspace{.1cm}
\item[(iii)] $A$ has strict comparison. 
\end{itemize}
\end{cor}

\begin{proof} The implication (ii) $\Rightarrow$ (i) follows from Theorem \ref{thm:(SI).implies.Z-absorbtion}  and from Corollary \ref{cor:from.excision.to.(SI).for.nuclear.A}. 

(i) $\Rightarrow$ (iii) holds in general, cf.\ \cite{Ror.Z.absorb}; and (iii) $\Rightarrow$ (ii) holds in general (for trivial reasons). 
\end{proof}

\noindent It seems plausible that the three conditions above are equivalent for all non-elementary
stably finite, separable, nuclear, simple and unital \Cs{} $A$, without the assumption on the trace simplex. 

The equivalence between (ii) and (iii) is curious. It may hold for all \emph{simple} \Cs s (but it  does not hold in general for non-simple \Cs s).

It follows from Corollary \ref{cor:nuclear.case}, together with Lemma \ref{lm:pure}, that any non-elementary
stably finite, separable, nuclear, simple and unital \Cs{} $A$, for which $\partial T(A)$ is closed in $T(A)$
and $\partial T(A)$  is a finite-dimensional topological space, and which is $(m,\bar{m})$-pure for some $m,\bar{m} \in \N$, satisfies $A \cong A \otimes \cZ$. This shows that the main theorem from Winter's paper, \cite{Winter.Z}, concerning \Cs s with locally finite nuclear dimension follows from Corollary \ref{cor:nuclear.case}  in the case where $\partial T(A)$ is closed in $T(A)$, and $\partial T(A)$  is a finite-dimensional topological space. This shows how important it is to resolve the question if the conditions on $T(A)$ in Corollary \ref{cor:nuclear.case} can be removed!


\bigskip

\noindent
\address{Institut f{\"u}r Mathematik,
Humboldt Universit{\"a}t zu Berlin,\\ 
Unter den Linden 6, 
D--10099 Berlin, Germany}\\
\email{kirchbrg@mathematik.hu-berlin.de}

\medskip

\noindent
\address{Department of Mathematical Sciences, 
University of Copenhagen,\\ 
Universitetsparken 5,
DK-2100 Copenhagen, Denmark}\\
\email{rordam@math.ku.dk}
\end{document}